\documentclass{amsart}
\usepackage{amssymb,verbatim,amscd,mathrsfs,stmaryrd,wasysym,latexsym, bbm}
\usepackage[all]{xy} 
\usepackage{cancel}
\usepackage{lscape}
\usepackage{amsmath,amsthm,amsfonts,amssymb,amscd,verbatim,latexsym}
\usepackage{hyperref}
\usepackage{xcolor}
\hypersetup{
    colorlinks,
    linkcolor={blue!80!black},
    citecolor={blue!80!black},
}
\usepackage{enumerate}
\usepackage{dsfont}
\usepackage{array}
\usepackage{color}
\usepackage{bbold}
\usepackage{bbm}

\usepackage{cleveref}
\usepackage{float}
\usepackage{caption} 
\newcommand{\V}{\mathcal{V}}

\DeclareMathOperator{\gldim}{gldim}
\DeclareMathOperator{\GKdim}{GKdim}

\DeclareMathOperator{\Ext}{Ext}

\DeclareMathOperator{\cov}{cov}

\newcommand{\bfl}{\mathfrak l}

\newcommand{\Hom}{\operatorname{Hom}}

\newcommand{\Spec}{\operatorname{Spec}}

\newcommand{\Proj}{\operatorname{Proj}}
\newcommand{\MCM}{\operatorname{MCM}}
\newcommand{\gr}{\operatorname{gr}}
\newcommand{\tors}{\operatorname{tors}}
\newcommand{\rank}{\operatorname{rank}}

\newcommand{\Z}{\mathbb{Z}}
\newcommand{\lrr}{l_{\Re}}
\newcommand{\NN}{\mathbb{N}}
\newcommand{\fp}{\mathfrak{p}}

\newcommand{\la}{\langle}
\newcommand{\ra}{\rangle}
\newcommand{\tensor}{\otimes}
\newcommand{\tsr}{\tensor}
\newtheorem{theorem}{Theorem}[subsection]
\newtheorem{lemma}[theorem]{Lemma}
\newtheorem{corollary}[theorem]{Corollary}
\newtheorem{proposition}[theorem]{Proposition}

\newtheorem{question}[theorem]{Question}
\theoremstyle{definition}
\newtheorem{definition}[theorem]{Definition}
\newtheorem{example}[theorem]{Example}

\newtheorem{notation}[theorem]{Notation}
\renewcommand{\i}{\mathbbm{i}}

\renewcommand{\k}{\mathbb K}
\newcommand{\coker}{\hbox{coker}}

\renewcommand{\l}{\lambda}

\newcommand{\G}{\Gamma}

\newcommand{\z}{\zeta}

\renewcommand{\t}{\tau}

\renewcommand{\P}{{\mathbb{P}}}

\newcommand{\Q}{{\mathbb{Q}}}

\newcommand{\mc}{\mathcal}

\newtheorem{remark}[theorem]{Remark}

\numberwithin{equation}{section}

\begin{document}

\title[AS regular algebras and their geometry]{Some Artin-Schelter regular algebras from dual reflection groups and their geometry}


%    author one information

\author{Peter Goetz}
\address{
California State Polytechnic University, Humboldt, 
Department of Mathematics,
Arcata, CA 95521}
\email{pdg11@humboldt.edu}

\author{Ellen Kirkman}
\address{Wake Forest University, Department of Mathematics and Statistics, P. O. Box 7388, Winston-Salem, North Carolina 27109} 
\email{kirkman@wfu.edu}

\author{W. Frank Moore}
\address{Wake Forest University, Department of Mathematics and Statistics, P. O. Box 7388, Winston-Salem, North Carolina 27109}
\email{moorewf@wfu.edu}

\author{Kent B. Vashaw}
\address{University of California Los Angeles, Department of Mathematics, Los Angeles, California 90095}
\email{kentvashaw@math.ucla.edu}
%    \subjclass is required.

\date{}

\dedicatory{}

%    Abstract is required.
\begin{abstract}
Let $G$ be a group coacting on an Artin-Schelter regular algebra $A$ homogeneously and  inner-faithfully. When the
identity component $A_e$ is also Artin-Schelter regular, providing a generalization of the Shephard-Todd-Chevalley Theorem, we say that $G$ is a dual reflection group for $A$.  We give two examples of dual reflection groups of order 16, and study algebraic and geometric properties of three associated Artin-Schelter regular algebras of dimension four. 
\end{abstract}

\maketitle
%%%%%%%%%%%%%%%%%%%%%%%%%%%%%   Ellen
\section{Introduction} 
%\textcolor{red}{added \cite{RZ, LPWZ, V, ST, SV, CV} to bib file}\\

%{\kv KV: my edits and comments throughout are in this color.}
%{\pg PG: my edits and comments throughout are in this color.}

Let $\k$ be an algebraically closed field of characteristic zero. 
The classical Shephard-Todd-Chevalley Theorem (\cite{Ch,ST}) states that if $G$ is 
a finite group acting faithfully as graded automorphisms of 
${\k}[x_1, \dots, x_n]$,
then the fixed subring ${\k}[x_1,\dots, x_n]^G$ is isomorphic 
to ${\k}[x_1,\dots, x_n]$ if and only if $G$ is a  reflection group, and
Shephard and Todd classified the indecomposable complex reflection groups. Generalizations of the Shephard-Todd-Chevalley Theorem to the case where the commutative polynomial ring is replaced by an Artin-Schelter regular (AS regular) algebra $A$ (see Definition \ref{def2.1}), and the finite group is replaced by a Hopf algebra $H$, have been considered in works such as \cite{BB,C,FKMW1,FKMW2,KKZ1,KKZ2,KKZ3,Ki}.  In this paper we consider the action by the Hopf algebra 
 $H:=\k^G:=\Hom_{\k}(\k G, \k)$, or 
equivalently  a $G$-coaction, which provides an AS regular algebra $A$ with a 
$G$-grading, $A = \bigoplus A_g$, and the fixed subring $A^H$  is the identity 
component $A_e$. When $G$ is a finite abelian group the dual of the group algebra is isomorphic to the group algebra, giving a group action.  Hence we will always assume the group $G$ is a finite nonabelian group.

AS regular algebras were introduced by Artin and Schelter 
\cite{AS} in the 1980's, and include skew polynomial rings, iterated Ore extensions of AS regular algebras, and most of the Sklyanin algebras.
An AS regular algebra that is commutative is isomorphic to a commutative polynomial ring, hence the context of the invariants $A^H$ of a Hopf algebra $H$ acting on an AS regular algebra $A$ extends the classical (co)-commutative setting. The AS regular algebras of dimensions two and  three have been classified (\cite{ATV1,ATV2}), and current work on classifying AS regular algebras focuses on classifying AS regular algebras of dimension four (e.g. \cite{CV,ChV,CKS,CS1,CS2,CSV,LPWZ,RZ,SS,TV,V}). 
\begin{definition}\cite[Definition {\rm{0.1}}]{KKZ3}
Recall that the action of a Hopf algebra $H$ on an algebra $A$ is called inner faithful if there is no Hopf ideal $0 \neq I$ in $H$ with $I A = 0$. A finite group $G$ is called a 
{\it dual reflection group for an algebra $A$} if the Hopf algebra  $H:=\k^G$ acts
homogeneously and inner faithfully on a noetherian AS regular 
domain $A$ generated in degree 1 such that the fixed subring 
$A^H$ is again AS regular, i.e., the identity component of $A$ under 
the $G$-grading is AS regular. In this case we say that 
{\it $G$ coacts on $A$ as a dual reflection group}.  
\end{definition}
 The study of dual reflection groups began in \cite{KKZ3}, where certain necessary conditions for $G$ 	to coact as a dual reflection group on $A$ were proven; these conditions are reviewed in Section 2  (Theorem \ref{necessary}).  As part of his Master's degree thesis, using these necessary conditions, the fourth author wrote a program in Magma to find other potential dual reflection groups of order 16 and their corresponding quadratic AS regular algebras that are domains.  This work produced two new dual reflection groups, the semidihedral group and the ``modular group" (called $M_{16}$ in Magma)  and their associated AS regular algebras, whose relations were obtained from conditions on the coaction of the group on $A$ necessary to produce an AS regular fixed ring $A_e$.  The modular and semidihedral  groups are two of the four nonabelian groups of order 16 that have a cyclic subgroup of index two. The groups determine the relations of $A$ up to coefficients on the generators; a second choice of coefficients on the algebra associated to the semidihedral group produced the third AS regular algebra. We denote by $R$ (Definition \ref{definition R}) the algebra associated to the modular group, by $S$ (Definition \ref{definition S}), and by $T$ (Definition \ref{definition T}) the algebras associated to the semidihedral group. We show that the algebra $R$ is AS regular in Proposition \ref{r-props} and that $S$ and $T$ are AS regular in Theorem \ref{regular}. The algebra $S$ was mentioned in \cite[Remark 4.2(d)]{TV} before a complete proof that it was AS regular had been given.

The AS regular algebras that arose from the two groups are also interesting as they are noetherian PI AS regular of dimension four (Corollary \ref{consequences of PI}) and,  due to the way they are constructed,  their relations are binomial relations.   Their geometry is also interesting. The AS regular algebras $S$ and $T$ have  a point scheme of twenty points 
(Theorem \ref{point schemes, automorphisms of S, T}) and a one-dimensional line scheme with ten components, each of degree two (Theorem \ref{line scheme of S, T}). Furthermore, we show that each component of the line scheme of the algebra $S$ or $T$ corresponds to a canonically defined smooth quadric in $\P^3$; of the two rulings on each of these quadrics, exactly one of the rulings corresponds to line modules (Theorem \ref{lines in P^3 for S, T}).
The AS regular algebra $R$ is a double Ore extension (Proposition \ref{r-props}), in the sense of J.J. Zhang and J. Zhang \cite{ZZ1, ZZ2}.  Its point scheme has six irreducible components, each isomorphic to a $\mathbb{P}^1$ embedded in $\mathbb{P}^3$ (Theorem \ref{point schemes, automorphisms of R}). Its line scheme has dimension two and consists of seven components that sit in $\mathbb{P}^5$. Four of the components correspond by the Pl\"{u}cker embedding to the space of all lines in $\mathbb{P}^3$ sitting in four hyperplanes in $\P^3$; the other three components are smooth quadrics in $\mathbb{P}^3$ (hence isomorphic to $\mathbb{P}^1 \times \mathbb{P}^1$) (Theorem \ref{line scheme of R}).  The geometry of the line schemes of $R,S$ and $T$ is different than that observed in the AS regular algebras of dimension four studied in \cite{CV, CTVW, CV2, CS1,CS2, CSV}. 

 The point schemes and line schemes of algebras have been used in classifying AS regular algebras.  Artin, Tate, and Van den Bergh used point schemes to classify AS regular algebras of dimension three \cite{ATV1}. Moreover, they showed that the point scheme is the graph of an automorphism, and the algebra is PI if and only if the automorphism has finite order. 
 In unpublished work in the mid-1990s Van den Bergh proved that any quadratic algebra with four generators and six generic relations has at most twenty nonisomorphic point modules (counted with multiplicity). If the algebra is also Auslander regular of global dimension four then the line scheme is one-dimensional (see \textsection 1D of \cite{V}). Hence the four-dimensional AS regular algebras with four generators and six relations, where the point scheme has twenty points and the line scheme has dimension one, are referred to as generic quadratic AS regular algebras; the first such algebra was constructed by Shelton and Tingey in \cite{ShelTing} and more examples using graded skew-Clifford algebras were constructed in \cite{CV}. In \cite{ShelTing} it is shown that the defining relations of a quadratic AS regular algebra of dimension four that has four generators, six relations, and a one-dimensional line scheme, satisfying some other homological conditions,  are determined by the line scheme. In \cite[Conjecture 4.2]{ChV} it is conjectured that the line scheme of the most generic quadratic AS regular algebra of dimension four with four generators and six relations is the union of two spatial elliptic curves and four planar elliptic curves.

 The paper is organized as follows. In \Cref{sect-prelim}, basic definitions and properties of graded rings and modules are given. Conditions on a group $G$ necessary to produce an AS regular algebra $A$ on which $G$ coacts with identity component AS regular are given, and the algebras $R$, $S$ and $T$ are defined. In \Cref{sect-groups-dual-ref} we prove that the algebras $R$, $S$ and $T$ are AS regular, that $R$ is graded by the modular group of order 16 with identity component AS regular, and that $S$ and $T$ are graded by the semidihedral group of order 16 with identity components AS regular. We also discuss further algebraic properties of $R$, $S$ and $T$. The geometric properties of $R$, $S$ and $T$ are given in \Cref{sect-geom}, where the point schemes, line schemes, and incidence relations are given. The paper concludes in \Cref{sect-questions} with problems for further study.

\subsection*{Acknowledgements}

We thank Michaela Vancliff for her interest in this work, and for particularly helpful comments on an early draft that substantially improved the organization of this paper. 

The fourth author was partially supported by NSF postdoctoral fellowship DMS-2103272 and by an AMS-Simons Travel Grant. This paper contains work supported by the National Science Foundation under Grant No. DMS-1928930, while the second author was in residence at the Simons-Laufer Mathematical Sciences  Institute in Berkeley, California, during the spring semester of 2024.

\section{Preliminaries} 
\label{sect-prelim}
Throughout $\k$ is an algebraically closed field of characteristic zero.

% {\kv Note on subsections: Section 2 has subsections, Sections 3-5 do not, and Section 6 has both subsections and sub-sub sections. For consistency, suggest making this consistent. For example, we could combine Sections 3-5 into a single section.\\}{\color{brown} Done 7/19 ek}
 \subsection{Dual reflection groups: necessary conditions}
 \label{subsect-dual-reflect} 
  A $\k$-algebra $A$ is called {\it connected graded} if
$$A=\k \oplus A_1\oplus A_2\oplus \cdots$$
for $A_i$ finite dimensional and $A_iA_j\subseteq A_{i+j}$ for all $i,j\in {\mathbb N}$.
The {\it Hilbert series} of $A$ is defined to be
$$H_A(t)=\sum_{i\in {\mathbb N}} (\dim_{\k} A_i)t^i.$$
The algebras that we use to replace commutative polynomial rings
are the AS regular algebras \cite{AS}. We recall the definition
below.

\begin{definition}
\label{def2.1}
A connected graded algebra $A$ is called {\it Artin-Schelter regular} 
(or {\it AS regular}) if the following conditions hold:
\begin{enumerate}
\item[(a)]
$A$ has global dimension dimension $d<\infty$ on
the left and on the right,
\item[(b)]
$\Ext^i_A(_A\k,_AA)=\Ext^i_{A}(\k_A,A_A)=0$ for all
$i\neq d$,
\item[(c)]
$\Ext^d_A(_A\k,_AA)\cong \Ext^d_{A}(\k_A,A_A)\cong \k(\bfl)$ for some
integer $\bfl$, and 
\item[(d)]
$A$ has finite Gelfand-Kirillov dimension.
\end{enumerate}
\end{definition}
 In addition, we will assume throughout that $A$ is a noetherian integral domain generated in degree one that has an additional grading by a finite group $G$ defined on its generating set.
 Throughout let $e$ denote the identity element of $G$. 

We will consider coactions on $A$ that are defined homogeneously on the generators of $A$ which are contained in $A_1$.  If the action of a semisimple Hopf algebra $H$ is inner faithful, then it follows from a result of Rieffel \cite{R} (see also \cite{FKMW1}) that all of the simple $H$-modules occur as summands in the decomposition of some tensor power of the $H$-module $A_1$.  For a group coaction that requirement means that the group grades of the generators of $A$ must be a generating set (but not necessarily a minimal generating set) of $G$.

In \cite{C}, Crawford completely classified the dual reflection groups for any noncommutative Artin-Schelter algebra $A$ of global dimension two.  He has shown the following.
\begin{example} \label{Craw}
 For a finite nonabelian group $G$ to %\textcolor{red}{(do we mean coact?)} 
 coact homogeneously and inner faithfully on $A$, then for some $m \geq 2,$ there exist $a$ and $b$ in $G$ satisfying $a^2=b^2$, and $A= \k \langle u, v \rangle/(u^2-v^2)$ with  grade $ u= a$ and grade $v= b$. The group will coact as a dual reflection group on $A$ if and only if we have $G = \langle a, b~|~ a^2 = b^2, a^{4m} = (ab)^{m} = e\rangle.$ The element $a^2$ has order $2m$ and is central in $G$, and $G/(a^2) \cong D_{2m}$, the dihedral group of order $2m$ and $|G| = 4m^2$. The subring of invariants under this coaction is the commutative polynomial ring $\k[(uv)^m,(vu)^m]$.
\end{example}

Next we recall some definitions from \cite{KKZ3}.
We say $\Re\subseteq G$ is a set of generators of $G$ if $\Re$ 
generates $G$ and $e\not\in \Re$ ($\Re$ need not be a minimal set of generators of $G$). 
\begin{definition} \cite[Definition 2.1]{KKZ3}
\label{def2.2}
Let $\Re$ be a set of generators of a group $G$.
The {\it length} of an element $g\in G$ with respect to 
$\Re$ is defined to be 
$$\lrr(g):=\min\{ n \mid v_1\cdots v_n=g, \; {\text{for some $v_i\in \Re$}}\}.$$
We define the length of the identity $e$ to be 0.  
\end{definition}
\begin{definition} \cite[Definition 2.2]{KKZ3}
\label{xxdef2.2}
Let $\Re$ be a set of generators of $G$.
The {\it Poincar{\'e} polynomial} associated to $\Re$ is defined to be
$$\fp_\Re(t)=\sum_{g\in G} t^{\lrr(g)}.$$
\end{definition}

This terminology was borrowed from the study of Coxeter groups.
The notion of length of a group element with respect to $\Re$, 
where $\Re$ is the set of Coxeter generators, is standard for 
Coxeter groups \cite[p.15]{BjB}, and its generating function 
is called the Poincar\'{e} polynomial in \cite[p. 201]{BjB}.  

\begin{theorem} \cite[Theorem 3.5]{KKZ3}
\label{necessary} Let $A$ be a noetherian AS regular domain
and $G$ a finite group that coacts on $A$ as a dual reflection group.
Let $H=\k^G$.   Then:
\begin{enumerate}[(1)]
%\item[(1)]
%There is a set of homogeneous elements $\{ f_{g}\mid g\in G\}\subseteq A$ 
%with $f_e=1$ such that $A_g=f_g \cdot A_e=A_e \cdot f_g$ for all $g\in G$. 
%As a consequence, the nonzero component of $A_g$ with lowest degree 
%has dimension 1.
\item
$I:=A (A_e)_{\geq 1}$ is a two-sided ideal and the covariant ring
$A^{\cov\; H}:=A/I$ is Frobenius.
%\item
%Suppose that, as an $H$-module, $A_1\cong
%\oplus_{g\in G} (\k p_g)^{n_g}$ where $n_g\geq 0$. If $n_g>0$
%and $g\neq e$, then $n_g=1$.
\item 
Let $n_g$ denote the dimension of $A_g \cap A_1$. The set $\Re:=\{g\in G\mid n_g>0, g\neq e\}$ generates $G$. 
\item
%%$A^{cov\; H}$ is a skew Hasse algebra associated to the generating set $\Re$.
%As a consequence, $A^{cov\; H}$ is a Frobenius algebra 
%generated in degree 1. Furthermore, 
The Hilbert series of
$A^{\cov\; H}$ is $p_{\Re}(t)$, the Poincar\'{e} polynomial of $G$ 
relative to $\Re$, and $p_{\Re}(t)$ is a product of cyclotomic
polynomials. 
\end{enumerate}
\end{theorem}
Note that $p_{\Re}(1)= |G|$, and if $\Phi_n(t)$ is the $n^\text{th}$ cyclotomic polynomial then \[  \Phi_n(1) = \left\{ \begin{array}{ccl}  p & & n = p^r~{\rm for}~p~{\rm prime} \\ 1 & & {\rm otherwise.}\end{array}\right. \]
\begin{example} \cite[Example 3.7]{KKZ3}
\label{ex2.5} The dihedral group $D_{8}$ of order 8  is generated by
$r$ of order two and $\rho$ of order four subject to the relation 
$r\rho=\rho^3 r$. Let $A$ be generated by 
$x,y,z$ subject to the relations
$$
zx = q xz,\quad
yx = a zy,\quad
yz = xy,
$$
%$$\begin{aligned}
%zx& = q xz,\\
%yx& = a zy,\\
%yz& = xy,
%\end{aligned}
%$$
for any $a\in \k^{\times}$ and $q^2=1$. It is easy to see
that $A=\k_{q}[x,z][y; \sigma]$, where $\sigma$ sends $z$ to $x$
and $x$ to $az$. Hence $A$ is an AS regular algebra of global dimension three
that is not PI (for generic $a$). Define the
$G$-degree of the generators of $A$ as
$$\deg_G(x)=r, \quad \deg_G(y)= r\rho, \quad \deg_G(z)=
r\rho^2.$$
The group $G$ coacts on 
$A$ homogeneously and inner faithfully, and 
$A_e=\k[x^2,z^2][y^2, \tau]$
where $\tau=\sigma^2\mid_{\k[x^2,z^2]}$. Therefore
$A_e$ is AS regular and $G$ is a dual reflection group.  
Here $\Re = \{r, r\rho, r\rho^2\}$, and the Poincar\'{e} 
polynomial is $p_{\Re}(t) = (1+t)^3$.

\end{example}
We do not know another $n$ with $D_{2n}$ 
a dual reflection group, and hence the class of dual 
reflection groups appears to be different from the 
class of reflection groups.  The dihedral group of
order 8 is an example of the following general construction.
\begin{example}
\label{wreath}
 Let
$W_i = \langle w_i \rangle $ be a cyclic group order $r$,
 $U = W_1 \times \cdots \times W_n$,
$V = \langle v \rangle $ a cyclic group of order $np$, and
$\phi:V \longrightarrow S_n$ be given by $\phi(v) = (1,2, \ldots , n)$.
Fix $q:= \pm 1$ and $R := \k_q[x_1, x_2, \ldots , x_n]$, and let
{$A := R[y;\tau]$}, where $\tau(x_i) := x_{i+1}$ for $i = 1, 2, \ldots , n-1$ and $\tau(x_n) :=x_1$. Let $G$ be the
semi-direct product {$U \rtimes_{\phi}V$} with 
  $(u_1v_1)(u_2v_2) =(u_1\phi_{v_1}(u_2))(v_1v_2)$.
Then {$A$} is {$G$}-graded by $\deg_G(x_i)= w_i$ and $ \deg_G(y)= v$.  
The invariant algebra is given by ${A}^{\k^G} = \k_Q[x_1^r, x_2^r, \ldots , x_n^r][y^{np}]$, where $Q=q^{r^2}$, an AS regular algebra.
\end{example}

Returning to case Example \ref{Craw} for $m=2$ we have the following:
\begin{example}
The independent monomials in $A = \k \langle u, v \rangle/(u^2-v^2)$ are of the form $u^i(vu)^jv^k$ for $k = 0$ or $1$.  In this case $G$ is the modular group of order 16 (generated by $a$  and $ab$ where $a^8 = (ab)^2 = e$, $(ab) a (ab)^{-1} = abaab=a^3 b^2 = a^5$ with and we can take
$\Re = \{a,b\}$, and find that:

\begin{center}
\begin{tabular}{|c|c|c|c|c|c|c|c|}
\hline
Length   & $0$ & $1$ & $2$ & $3$ & $4$ & $5$ & $6$ \\ \hline
Elements & $e$ & $a,b$ & $a^2,ab,ba$ & $a^3, a^2b, aba, bab$
         & $a^4, a^3b, a^2ba$ &  $a^5, a^4b$ & $a^6$ \\ \hline
\end{tabular}
\end{center}
Therefore, the Poincar\'{e} polynomial is $$\fp_\Re(t )= 1 +2t+3t^2 + 4t^3 + 3t^4+2t^5 +t^6 = (t+1)^2(t^2+1)^2,$$
and the Hilbert series of the fixed ring $A^H = \k[(uv)^2, (vu)^2]$ is  
\[
H_{A^H}(t) = H_A(t)/ \fp_\Re(t ) = 1/((1-t)^4(t+1)^2(t^2+1)^2) = 1/(1-t^4)^2.
\]
\end{example}

Based on the necessary properties of a dual reflection group given in Theorem \ref{necessary}, the fourth author approached a search for dual reflection groups in the following way. He wrote a program in the computer algebra system Magma that for each nonabelian group $G$ of order 16 found all possible generating sets for the group, and then computed their Poincar\'{e} polynomials.  For each Poincar\'{e} polynomial that was a product of cyclotomic polynomials, he wrote out a possible multiplication table for an algebra whose degree one generators followed that group grading so that Theorem \ref{necessary} was satisfied.  From this table a relation was formed equating those degree 2 products of the same group grade so that the resulting algebra would be graded by $G$. Only algebras with quadratic relations were considered.  This process led to some examples such as Example \ref{wreath} and to the three new examples that we introduce next.

We illustrate how the algebras were originally defined using the method described above.  We first consider the algebra $R$ related to the modular group $M_{16}$; the construction of the algebra $S$ related to the semidihedral group $SD_{16}$ is similar. The algebra $T$ was constructed by considering scalar coefficients in the relations of $S$, which preserve the grading, while maintaining the AS regularity of the algebra.  When we prove properties of these algebras in a subsequent paper we will use more conventional presentations of the groups.  The group presentations we use here are the ones found in Magma.

We define the algebra $R$ from the group of order 16  
that  Magma calls $M_{16}$ and is often called the ``modular group" because its lattice of subgroups is a modular lattice.
 Magma presents $M_{16}$ as
$$M_{16}:= \langle a,b,c,d \; | \; a^2=c, b^2 = d^2 =e, c^2 = d, a^{-1}ba = bd, \mbox{ with } c,d \mbox{ central} \rangle.$$

To define the algebra  $R$ on which $M_{16}$ coacts  we use the generating set  $\Re = \{a, acd, ab, abc \}$ which has
Poincar\'{e} polynomial 
 $\fp_\Re(t ) = (1+t)^4$, a product of cyclotomic polynomials. The products of all pairs of these elements is:
\begin{center}
\begin{tabular}{|cc||c|c|c|c|}
\hline
%&& & &   & \\
&& {$a$} & {$acd$} & {$ab$}& {$abc$}\\
\hline
\hline
 & {$a$} & {$c$} & {$e$} & {$bc$} &{$bd$}\\
\hline
 &{$acd$}&{$e$}& {$cd$} &{$b$} & {$bc$}\\
\hline
 & {$ab$}&{$bcd$} &{$bd$} &{$cd$} &{$e$}\\
\hline
 &{$ abc$} &{$b$} &{$bcd$} &{$e$} &{$c$}\\
\hline
\end{tabular}
\end{center}
Hence we define an algebra with four generators $x_1, x_2, x_3, x_4$ and assign the generators the
group grades: grade $x_1 = a$, grade $x_2 = acd$,  grade $x_3 = ab$, and grade $x_4 =abc$.  To obtain
a graded algebra with six binomial relations we identify the elements in the non-identity grades.
This gives us the algebra  $R$.
\begin{definition}\label{definition R}
Let $$R = \frac{\k \la x_1, x_2, x_3, x_4 \ra}{\la x_1^2-x_4^2, x_2^2-x_3^2, x_4x_1-x_2x_3, x_1x_4-x_3x_2, x_1x_3-x_2x_4,x_3x_1-x_4x_2 \ra}.$$  Grade 
$R$  with the generating set of the modular group $\Re = \{a, acd, ab, abc\}$ assigning group grades:
$$\text{grade } x_1 = a, \quad \text{grade } x_2 = acd, \quad \text{grade } x_3 =  ab, \quad \text{grade } x_4  = abc.$$
The grades of the relations of $R$ as follows:
%$$x_1^2 - x_4^2 \;(\text{grade } c),  \;\; x_2^2 - x_3^2 \; (\text{grade } cd),$$
%$$x_4x_1 - x_2 x_3 \;(\text{grade } b), \;\;x_1x_4 - x_3 x_2 \;(\text{grade } bd),$$
%$$x_1x_3 - x_2x_4 \;(\text{grade } bc),\;\; x_3x_1 - x_4x_2 \;(\text{grade } bcd).$$
$$\begin{array}{ll}
x_1^2 - x_4^2 \,\;\qquad(\mbox{grade } c),  &
x_2^2 - x_3^2 \,\;\qquad(\mbox{grade } cd), \\
x_4x_1 - x_2 x_3       ~(\mbox{grade } b),  &
x_1x_4 - x_3 x_2       ~(\mbox{grade } bd), \\
x_1x_3 - x_2x_4        ~(\mbox{grade } bc), &
x_3x_1 - x_4x_2        ~(\mbox{grade } bcd).
\end{array}$$

\end{definition}
In Proposition \ref{r-props} we will show that $R$ can be expressed as a double Ore extension, and later we will use the presentation of $R$ given in the proof of that proposition.

Next we define the algebra $S$ from the group of order 16 that  Magma calls  the ``semidihedral group" and  presents as
$$SD_{16}:= \langle a,b,c,d \; | \; a^2=c^2=d, b^2 = d^2 =e, $$$$ a^{-1}ba = bc, \; a^{-1}ca = cd,\; b^{-1}cb = cd,\mbox{ with } d \mbox{ central} \rangle.$$
To define the algebra  $S$ on which $SD_{16}$ coacts  we use the generating set  $\Re = \{b, bc, ab, abcd \}$ which has
Poincar\'{e} polynomial 
 $\fp_\Re(t ) = (1+t)^4$, a product of cyclotomic polynomials. 
\begin{definition} \label{definition S}
For the semidihedral $SD_{16}$, let
$$S = \frac{\k \la x_1, x_2, x_3, x_4 \ra}{\la x_1x_2-x_3^2, x_4^2-x_2x_1, x_1x_3-x_2x_4, x_4x_1-x_3x_2, x_2x_3-x_3x_1, x_4x_2-x_1x_4 \ra}.$$  
We endow $S$ with an $SD_{16}$-grading via $$ \text{grade } x_1 = b,\quad \text{grade } x_2 = bc,\quad  \text{grade } x_3 = ab,\quad  \text{grade } x_4 = abcd.$$
The grades of the relations of $S$ are given by:
%$$x_3^2 - x_1x_2 \mbox{ (grade $c$)} , \;\; x_4^2 - x_2x_1 \mbox{ (grade $cd$),}$$
%$$x_1x_3 - x_2 x_4 \mbox{ (grade $acd$)} , \;\;x_3x_1 - x_2 x_3 \mbox{ (grade $a$),} $$
%$$x_1x_4 - x_4x_2 \mbox{ (grade $ad$)} ,\;\; x_4x_1 - x_3x_2 \mbox{ (grade $ac$)} .$$
$$\begin{array}{ll}
x_3^2 - x_1x_2  \;\quad\mbox{(grade $c$)},   &
x_4^2 - x_2x_1   \;\quad\mbox{(grade $cd$)},  \\
x_1x_3 - x_2 x_4 ~\mbox{(grade $acd$)}, &
x_3x_1 - x_2 x_3 ~\mbox{(grade $a$),}   \\
x_1x_4 - x_4x_2  ~\mbox{(grade $ad$)},  &
x_4x_1 - x_3x_2  ~\mbox{(grade $ac$)}.
\end{array}$$
\end{definition}
After investigating if adding scalar coefficients to these relations would produce a family of AS regular algebras, we found only the following algebra $T$, whose relations differ from those of $S$ in two signs.
\begin{definition}\label{definition T}
Let $$T = \frac{\k \la x_1, x_2, x_3, x_4 \ra}{\la x_1x_2-x_3^2, x_4^2
+x_2x_1, x_1x_3-x_2x_4, x_4x_1-x_3x_2, x_2x_3-x_3x_1, x_4x_2+x_1x_4 \ra}.$$  Grade 
$T$  with the generating set $\Re = \{b, bc, ab, abcd\}$ assigning group grades:
$$\text{grade } x_1 = b, \quad \text{grade } x_2 = bc, \quad \text{grade } x_3 =  ab, \quad \text{grade } x_4  = abcd.$$
The grades of the relations of $T$ are as follows:
$$\begin{array}{ll}
x_3^2 - x_1x_2 \;\quad\mbox{(grade $c$)}, &
x_4^2 + x_2x_1 \;\quad\mbox{(grade $cd$)}, \\
x_1x_3 - x_2 x_4~\mbox{(grade $acd$)}, &
x_3x_1 - x_2 x_3~\mbox{(grade $a$)}, \\
x_1x_4 +x_4x_2~\mbox{(grade $ad$)}, &
x_4x_1 - x_3x_2~\mbox{(grade $ac$)}.
\end{array}$$
\end{definition}

Next, we present some tools for studying
$R,S$ and $T$.
%%%%%%%%%%%%%%%%%%%%%%%%%%%%%%%%%%%%%%%5

% Pete
%\section{Preliminaries (Tools?)}

\subsection{Graded algebras as modules over subalgebras}

The following result has a straightforward proof based on the graded Nakayama lemma. We omit the proof.

\begin{lemma}\label{new f.g.}
Let $A$ be an $\NN$-graded ring, and let $Y = \{y_i\}_{i = 1}^n \subset A$ be a set of homogeneous elements of positive degree. Let $B$ be the subring of $A$ generated by $Y$. Let $I$, respectively $J$, denote the two-sided ideal of $A$, respectively $B$, generated by $Y$. If $I = JA$ and $A/I$ is finitely generated as a left $B/J$-module, then $A$ is finitely generated as a left $B$-module. More precisely, if $\{f_i+I : 1 \leq i \leq k, \, f_i \in A_{d_i}\} \subset A/I$ is a generating set for $A/I$ as a left $B/J$-module, then $\{f_i : 1 \leq i \leq k\}$ is a generating set for $A$ as a left $B$-module. 

If $I = AJ$, a similar statement holds for right modules, mutatis mutandis.
\end{lemma}

% {\color{blue} (PG: We could probably omit the following proof and just say it's an easy argument using the graded Nakayama lemma.)}

% \begin{proof}

% Suppose that $I = JA$ and assume that $\{f_i+I : 1 \leq i \leq k, \, f_i \in A_{d_i}\} \subset A/I$ is a generating set for $A/I$ as a left $B/J$-module. Let $A' = \sum_{i = 1}^k B f_i$. We need to prove that $A = A'$.

% Let $M = A/A'$ and consider $M$ as a left $B$-module in the obvious way. We claim that $JM = M$. Obviously $JM \subseteq M$. Conversely, let $a+A' \in M$ for some $a \in A$. By assumption, there exist $b_1, \ldots, b_k \in B$ such that $$a + I = \sum_{i = 1}^k (b_i+J)(f_i+I) + I = \sum_{i=1}^k b_i f_i + I.$$ Hence $a = \sum_{i=1}^k b_i f_i + z$, for some $z \in I$, and therefore $a + A' = z+A'$. By assumption, $I = JA$, so $z \in JA$, and it follows that $a+A' \in JM$. We have proved that $JM = M$.

% Finally, since $J$ is generated by homogeneous elements of $B$ of positive degree, the graded Nakayama lemma implies that $M = 0$. The result follows.
% \end{proof}

We will use the next result to determine the structure of the group-graded identity component of the algebras $R$, $S$ and $T$.

\begin{theorem}\label{free modules}
Let $A$ be a finitely generated, $\NN$-graded, $\k$-algebra with Hilbert series $H_A(t) = p(t)/(1-t)^n$, where $p(t) \in \k[t]$, $p(1) \neq 0$, $n \geq 1$. Let $Y = \{y_i\}_{i = 1}^n \subset A$ be a set of homogeneous elements of positive degree, say $\deg(y_i) = d_i > 0$. Let $B$ be the $\k$-subalgebra of $A$ generated by $Y$. Let $I$, respectively $J$, denote the two-sided ideal of $A$, respectively $B$, generated by $Y$. Assume that $I = JA$, $A/I$ is a finite-dimensional vector space over $B/J\cong \k$, and $H_{A/I}(t) = \prod_{i}(1-t^{d_i})H_A(t)$. Let $C$ be an $\NN$-graded, $\k$-algebra that is a domain with Hilbert series $H_{C}(t) = \prod_{i} (1-t^{d_i})^{-1}$, and assume that there is a graded algebra surjection $\pi: C \to B$. Then

\begin{itemize}
\item[(1)] $\pi$ is an isomorphism;
\item[(2)] $A$ is a free left $B$-module of rank $\dim_{\k} A/I$.
\end{itemize}

Similarly, if $I = AJ$, $A/I$ is a finite-dimensional vector space over $B/J\cong \k$,  and $H_{A/I}(t) = \prod_{i}(1-t^{d_i})H_A(t)$, then (1) holds and $A$ is a free right $B$-module of rank $\dim_{\k} A/I$.
\end{theorem}

\begin{proof}
First, since we are assuming that $I = JA$ and $A/I$ is a finite-dimensional vector space over $B/J\cong \k$, Lemma \ref{new f.g.} implies that $A$ is finitely generated as a left $B$-module. The assumption on the Hilbert series of $A$ implies that $\GKdim A = n$. Furthermore, since $A$ is a finitely generated $B$-module, \cite[Proposition 8.2.9]{MR} shows that $\GKdim B = n$. The assumption on the Hilbert series of $C$ also guarantees that $\GKdim C = n$. Consider the graded algebra surjection $\pi: C \to B$ and let $I = \ker \pi$. If $I \neq 0$, then, since $C$ is a domain, any nonzero element of $I$ is left regular. So \cite[Proposition 3.15]{KL} implies that $$\GKdim B = \GKdim C/I \leq \GKdim C - 1 = n-1,$$ a contradiction. Therefore $\pi$ is an isomorphism.

Now we prove (2). Since $B \cong C$, we have $H_B(t) = \prod_i (1-t^{d_i})^{-1}$. Let $r = \dim_{\k} A/I$ and let $u_1, \ldots, u_r \in A/I$ be a homogeneous $\k$-basis. Choosing homogeneous preimages $v_1, \ldots v_r \in A$, the proof of Lemma \ref{new f.g.} shows that $v_1, \ldots, v_r$ generate $A$ as a left $B$-module. Let $V$ denote the $\k$-linear space spanned by $v_1, \ldots, v_r$. Then the canonical multiplication map $B \tsr V \to A$ is surjective. Since
$$H_{B \tsr V}(t) = H_B(t)H_V(t) = H_{B}(t) H_{A/I}(t) =  \prod_i (1-t^{d_i})^{-1} \prod_i (1-t^{d_i}) H_A(t) = H_A(t),$$$A$ is a free left $B$-module with basis $\{v_1, \ldots, v_r\}$, and (2) follows.
\end{proof}

We will use the next two results to show that the algebras $R$, $S$ and $T$ are PI algebras. The following proposition is well known.

\begin{proposition}\label{normal regular sequences via Hilbert series}\cite[Lemma II.1.5]{G}
Let $A$ be a connected graded $\k$-algebra. Suppose that $y_1, \ldots, y_n$ are normal homogeneous elements of $A$, say $\deg(y_i) = d_i$. Let $A' = A/\la y_1, \ldots, y_n \ra$. Then $y_1, \ldots, y_n$ is a regular sequence if and only if $$H_{A'}(t) = \prod_{i=1}^n (1-t^{d_i}) \cdot H_A(t).$$ 
\end{proposition}

 The following result follows easily from Theorem \ref{free modules} and Proposition \ref{normal regular sequences via Hilbert series}.

% {\kv This corollary does not appear to be cited anywhere in the paper. Is it going to be used for something that hasn't been added yet? If it isn't relevant to anything later in the paper, perhaps we consider cutting.}

% {\pg I agree with Kent that we should cut the following corollary. We might want to state it in the center paper in the future.}
\begin{corollary}\label{central regular sequence implies module-finite}
 Let $A$ be a finitely generated $\NN$-graded $\k$-algebra such that $H_A(t) = p(t)/(1-t)^n$ for some $p(t) \in \k[t]$, $p(1) \neq 0$ and $n \in \NN$. Let $Z(A)$ denote the center of $A$. Suppose that $c_1, \ldots, c_n \in Z(A)$ is a regular sequence of homogenous elements of positive degree. Let $A' = A/\la c_1, \ldots, c_n \ra$ and let $C = \k[c_1, \ldots, c_n]$. Then:
 \begin{itemize}
 \item[(1)] $A$ is module-finite over $Z(A)$;
 \item[(2)] $C$ is a polynomial ring, i.e., the $c_i$ are algebraically independent over $\k$;
 \item[(3)] $A'$ is finite-dimensional over $\k$;
 \item[(4)] $A$ is a free $C$-module of rank equal to $\dim_{\k} A'$.
 \end{itemize}
 \end{corollary}

% Ellen
\section{The Groups \texorpdfstring{$M_{16}$}{M16} and \texorpdfstring{$SD_{16}$}{SD16} are Dual Reflection Groups}
\label{sect-groups-dual-ref}

%\subsection{The Modular Group of Order 16}

\subsection{The modular group \texorpdfstring{$M_{16}$}{M16} is a dual reflection group}

Recall that  $$R = \frac{\k \la x_1, x_2, x_3, x_4 \ra}{\la x_1^2-x_4^2, x_2^2-x_3^2, x_4x_1-x_2x_3, x_1x_4-x_3x_2, x_1x_3-x_2x_4,x_3x_1-x_4x_2 \ra}.$$

We first discuss properties of the algebra $R$.
\begin{proposition}
\label{r-props}
The algebra $R$ is a double Ore extension of an AS regular algebra of dimension two, and consequently $R$ is AS regular of dimension four, strongly noetherian, Auslander regular, Cohen-Macaulay, Koszul, and a domain. 
\end{proposition}
\begin{proof} Consider the following linear change of variables:
%(found using a computer \textcolor{red}{(Include this parenthetical?)}):
$$
\begin{array}{ll}
Y_1 = x_1+x_2+x_3+x_4, &
Y_2 = -x_1+x_2-x_3+x_4, \\ 
Z_1 = -x_1+x_2+x_3-x_4, &
Z_2 = -x_1-x_2+x_3+x_4.
\end{array}
$$
One checks that the following are the relations of $A$ in terms of $Y_1, Y_2, Z_1, Z_2$:
$$\begin{array}{lll}
Y_2 Y_1 = -Y_1 Y_2, & 
Z_2 Z_1 = - Z_1 Z_2, &
Y_1Z_1 = -Z_2Y_2, \\
Y_2Z_1 = Z_2Y_1, &
Y_1 Z_2 = Z_1 Y_2, &
Y_2Z_2 = -Z_1Y_1.
\end{array}$$
Using \cite[Definition 1.3]{ZZ1} %(and reversing the roles of $A$ and $B$ in that definition) 
we check that $R$ is a right double Ore extension of $B$, where $B$ is the $\k$-subalgebra generated by $Y_1$ and $Y_2$.  Indeed, 
(1) $R$ is generated by $B$ and $Z_1, Z_2$,
(2) $Z_1$ and $Z_2$ satisfy the relation $Z_2Z_2 = - Z_1 Z_2$,
(3)  As a left $B$-module $R = \sum B Z_1^{n_1}Z_2^{n_2}$, and it is a free $B$-module with basis $\{ {Z_1}^{n_1}{Z_2}^{n_2}\; | \;\ n_1 \geq 0, n_2 \geq 0 \},$ and
(4)  $Z_1B + Z_2 B \subset BZ_1 + B Z_2$.

Similarly $R$ is also a left double Ore extension of $B$. Since $B \cong \k_{-1}[Y_1,Y_2]$ is AS regular, $R$ is AS regular of global dimension four by \cite[Theorem 0.2]{ZZ1}.   The other properties hold by \cite[Theorem 0.1]{ZZ2}.
\end{proof}

We now check that via $R$, the group $M_{16}$ is a dual reflection group.

\begin{proposition}
\label{invar-r}
The invariant subalgebra $R_e$ is the subalgebra of $R$ generated by the set $Y=\{x_1x_2, x_2x_1, x_3x_4, x_4x_3\}$, and is isomorphic to $\mathbb{K}[z_1,\dots, z_4]$, the polynomial ring in four variables.    
\end{proposition}

\begin{proof}
    The proof will be an application of \Cref{free modules}. Set the following notation to match that of the theorem: let $B$ be the subalgebra of $R$ generated by $Y$, let $I$ be the ideal of $R$ generated by $Y$, let $J$ be the ideal of $B$ generated by $Y$, and let $C$ be the polynomial ring $\k[z_1,\dots, z_4]$ in four variables. Denote the elements of $Y$ by $y_1,\dots, y_4$. It can be checked using the relations of $R$ that the elements of $Y$ commute, so that we have a ring homomorphism $\pi: C \to B$ given by
    $\pi(z_i) = y_i$.

     We can additionally check that $I=JR = RJ$ by showing that for each $i,j=1, \dots, 4,$ we have $x_i y_j = y_k x_i$ for some $k=1,\dots, 4$. A computation using the diamond lemma shows that the images of the following elements constitute a $\mathbb{K}$-basis for $R/I$: 
    %\begin{enumerate}
    %    \item In degree 0: 1;
    %    \item In degree 1: $x_1, x_2, x_3, x_4$; 
    %    \item In degree 2: $x_2x_4$, $x_3x_2$, $x_3^2$, $x_4x_1$, $x_4x_2$, $x_4^2$;
    %   \item In degree 3: $x_3^2, x_4x_2x_4, x_4^2 x_1, x_4^3$;
    %    \item In degree 4: $x_4^4$. 
    %\end{enumerate}
    $$\{
    1,
    x_1, x_2, x_3, x_4,
    x_2x_4,x_3x_2,x_3^2,x_4x_1,x_4x_2,x_4^2,
    x_3^2,x_4x_2x_4,x_4^2x_1,x_4^3,
    x_4^4
    \}.
    $$
    That is, the Hilbert series of $R/I$ is $(1+t)^4$. Since the Hilbert series of $R$ is $(1-t)^{-4}$ by \Cref{r-props}, we have $H_{R/I}(t)=(1-t^2)^4 H_{R}(t),$ and hence all assumptions of \Cref{free modules} are satisfied. It follows that $\pi: C \to B$ is an isomorphism, that $R$ is a free $B$-module of rank 16, and the elements in (1)-(5) above constitute a $B$-basis for $R$. It is clear that $B \subseteq R_e$; to complete the proof, just note that all non-unital elements of the $B$-basis of $R$ given above are in nontrivial $M_{16}$-degree, and hence $R_e = B . 1 =B.$
\end{proof}
 The presentation of $R$ as a double Ore extension, given in Proposition \ref{r-props}, is more suited to proving the algebraic and geometric properties of $R$.  Hence we will use 
the presentation in terms of generators $Y_1, Y_2, Z_1, Z_2$ that is given in the proof of Proposition \ref{r-props}.  We note that this presentation is no longer graded homogeneously by $M_{16}$ in these generators. Henceforth, we will present $R$ as:
$$\frac{\k \la x_1, x_2, x_3, x_4 \ra}{\la x_1x_2+x_2x_1, x_1x_3+x_4x_2, x_1x_4-x_3x_2, x_2x_3-x_4x_1, x_2x_4+x_3x_1, x_3x_4+x_4x_3\ra}.$$

%\subsection{The Semi-Dihedral Group of Order 16}

\subsection{The semidihedral group of order 16 is a dual reflection group}

Recall the definition of the algebras $S$ and $T$:
\begin{center}
$S = \dfrac{\k \la x_1, x_2, x_3, x_4 \ra}{\la x_1x_2-x_3^2, x_4^2-x_2x_1, x_1x_3-x_2x_4, x_4x_1-x_3x_2, x_2x_3-x_3x_1, x_4x_2-x_1x_4 \ra},$ \\ \phantom{S} \\
$T = \dfrac{\k \la x_1, x_2, x_3, x_4 \ra}{\la x_1x_2-x_3^2, x_4^2
+x_2x_1, x_1x_3-x_2x_4, x_4x_1-x_3x_2, x_2x_3-x_3x_1, x_4x_2+x_1x_4 \ra}$.
\end{center}

We begin by discussing some of the properties of 
the algebra $S$. We first prove that $S$ is an 
Artin-Schelter regular algebra of dimension four.
The proof in the case of $T$ is similar and left to the
reader.

%; or more specifically a quantum projective 3-space. 

\begin{lemma} \label{G-basis} The following elements constitute a Gr\"obner basis for the defining ideal of $S$: $$x_1x_2-x_3^2,\; x_4^2-x_2x_1,\; x_1x_3-x_2x_4,\; x_4x_1-x_3x_2,\; x_2x_3-x_3x_1,\; x_4x_2-x_1x_4,$$ $$ x_4x_3^2-x_3x_2^2, \;x_4x_3x_2-x_2x_1^2,\; x_4x_3x_1-x_1x_4x_3.$$
\end{lemma}

\begin{proof} We order the variables as $x_3< x_2 < x_1 < x_4$ and order monomials using left lexicographical order. Then the statement follows from Bergman's diamond lemma and straightforward calculation.
\end{proof}

\begin{lemma}\label{monomial basis} There is a monomial basis for $S$ which consists of the monomials $$x_3^i x_2^j x_1^k (x_4x_3)^l \text{ and } x_3^i x_2^j x_1^k (x_4x_3)^l x_4, \ \ \ i, j, k, l \geq0.$$
\end{lemma}

\begin{proof} From Lemma \ref{G-basis} we know there is a monomial basis consisting of all monomials which do not contain any of $x_1x_2$, $x_4^2$, $x_1x_3$, $x_4x_1$, $x_2x_3$, $x_4x_2$, $x_4x_3^2$, $x_4x_3x_2$, or $x_4x_3x_1$ as submonomials. Suppose $w$ is such a monomial. 

First, if $w = w'x_2$, then it is clear that $w'$ must be of the form $x_3^i x_2^j$ for some $i, j \geq 0$. Second, if $w = w'x_1$, then it is clear that $w'$ must be of the form $x_3^i x_2^j x_1^k$ for some $i, j, k \geq 0$. Now we have established a complete description of monomials that end in $x_2$ or $x_1$. 

Third, suppose $w = w'x_3$. If $w'$ contains no occurrences of $x_1$ or $x_2$, then it is clear that $w$ must be of the form $(x_4x_3)^l$ for some $l \geq 1$. If $w$ contains $x_1$ as a submonomial, consider the first occurrence of such an $x_1$ reading right to left. Our description of the monomials ending in $x_1$ then forces  $w$ to be of the form $x_3^ix_2^j x_1^k (x_4x_3)^l$, for some $l \geq 1$. Similarly, if $w$ contains no occurrences of $x_1$ as a submonomial, but $w$ does contain $x_2$ as a submonomial, then $w$ must be of the form $x_3^i x_2^j (x_4x_3)^l$ for some $l \geq 1$.

Finally, suppose $w = w'x_4$. Notice that $w'$ cannot end in $x_4$. Then our descriptions of monomials ending in $x_3$, $x_2$ or $x_1$ applied to the monomial $w'$ imply that $w$ is of the form $x_3^i x_2^j x_1^k (x_4x_3)^l x_4$. 
\end{proof}
We say that an element of $S$ written in terms of 
this basis is in \emph{normal form}.
For later use we note the following obvious consequence of the last result.

\begin{corollary}\label{c regular} The element $x_3 \in 
S$ is left regular, i.e., for $f \in A$, $$x_3 f= 0 \implies f = 0.$$
\end{corollary}

The element $x_4 \in A$ is also left regular since it is easy to see that the assignments $$x_1 \mapsto x_2,\;  x_2 \mapsto x_1,\; x_3 \mapsto x_4,\; x_4 \mapsto x_3$$ extend to an automorphism of $A$.

\begin{proposition}\label{Hilbert series} The Hilbert series of $S$ is $(1-t)^{-4}$.
\end{proposition}

\begin{proof} Fix some integer $D \geq 0$. We now count all monomials $w$ as described in Lemma \ref{monomial basis} of total degree $D$. Let $a_D$ be the number of such monomials.

First recall that the number of monomials of the form $x_3^i x_2^j x_1^k$ such that $i, j, k \geq 0$ and $i+j+k = m$ for some fixed $m \geq 0$ is the binomial coefficient $\binom{m+2}{2}$. Now consider a monomial of the form $x_3^i x_2^j x_1^k (x_4x_3)^l$, $i+j+ k +2l = D$, or $x_3^i x_2^j x_1^k (x_4x_3)^l x_4$, $i+j+k+2l+1 = D$. Notice that $l$ ranges from $0$ to $\lfloor \frac{D}{2} \rfloor$. For the monomials of type $x_3^i x_2^j x_1^k (x_4x_3)^l$, for fixed $l$, we have $\binom{D-2l+2}{2}$ such monomials; for the monomials of type $x_3^i x_2^j x_1^k (x_4x_3)^l x_4$, for fixed $l$, we have $\binom{D-2l+1}{2}$ such monomials. Therefore $$a_D = \binom{D+2}{2} + \binom{D+1}{2} + \cdots + \binom{3}{2} + \binom{2}{2}.$$ The last expression is easily seen to equal $\binom{D+3}{3}$. We conclude that the Hilbert series of $S$ is $(1-t)^{-4}$.  
\end{proof}

The following lemma has an easy proof that we leave to the reader.

\begin{lemma}\label{5-lemma} Let $$0 \xrightarrow{} V_1 \xrightarrow{} V_2 \xrightarrow{} V_3 \xrightarrow{} V_4 \xrightarrow{} V_5 \xrightarrow{} 0$$ be a complex of finite dimensional $\k$-vector spaces which is exact at $V_1$, $V_2$, $V_4$ and $V_5$. If $\dim V_3 = \dim V_2 + \dim V_4 - \dim V_1 - \dim V_5$, then the complex is exact at $V_3$.
\end{lemma}

\begin{theorem}\label{Koszul} The algebra $S$ is Koszul. 
\end{theorem}

\begin{proof} 
Consider the following sequence of graded, right $S$-modules:
$$0 \xrightarrow{} S(-4) \xrightarrow{M_4} S(-3)^4 \xrightarrow{M_3} S(-2)^6 \xrightarrow{M_2} S(-1)^4 \xrightarrow{M_1} S \xrightarrow{} \k_S \xrightarrow{} 0,$$
where 
\begin{center}
\scalebox{.95}{
\begin{tabular}{ll}
$M_1 = \begin{bmatrix} x_1 & x_2 & x_3 & x_4 \end{bmatrix}$, &
$M_2 = \begin{bmatrix} x_2 & 0 & x_3 & 0 & 0 & x_4 \\ 0 & x_1 & -x_4 & 0 & x_3 & 0 \\ -x_3 & 0 & 0 & x_2 & -x_1 & 0 \\ 0 & -x_4 & 0 & -x_1 & 0 & -x_2 \\ \end{bmatrix}$, \\
$M_3 = \begin{bmatrix} 0 & x_1 & -x_3 & 0 \\ x_2 & 0 & 0 & -x_4 \\ 0 & 0 & x_1 & -x_2 \\ -x_4 & x_3 & 0 & 0 \\ -x_3 & 0 & x_2 & 0 \\ 0 & -x_4 & 0 & x_1 \end{bmatrix}$, &
$M_4 = \begin{bmatrix} x_1 \\ x_2 \\ x_3 \\ x_4 \\ \end{bmatrix}$.
\end{tabular}}
\end{center}

It is easy to check that this is a complex. Moreover, it is well known to be 
exact at $S(-1)^4$ and that $\coker M_1 = \k_A$.
Further, \Cref{c regular} shows
that it is exact at $S(-4)$.  To complete the proof
of exactness of the complex, by \Cref{5-lemma} it suffices show that it is exact at $S(-3)^4$. 

Suppose that $\begin{bmatrix} f_1 & f_2 & f_3 & f_4 \end{bmatrix}^T \in \ker M_3,$ where we assume that $f_1$, $f_2$, $f_3$, $f_4$ are written in normal form. Considering the fifth row of $M_3$, we see that $-x_3f_1 + x_2f_3 = 0$, so $x_2f_3 = x_3f_1$. Since $x_3 f_1$ is in normal form, it follows from the Gr\"obner basis in Lemma \ref{G-basis} that $f_3 = x_3t$ for some $t \in A$. Then we have $x_2x_3t = x_3f_1$, so $x_3x_1t = x_3f_1$, and since $x_3$ is left regular, we have $f_1 = x_1t$. Next, consider the fourth row of $M_3$; we have $-x_4f_1 + x_3f_2 = 0$, so $x_4f_1 = x_3f_2$, and then $x_4x_1t = x_3 f_2$, therefore $x_3x_2t = x_3 f_2$. It follows that $s_2 = x_2t$. Last, consider the second row of $M_3$; we have $x_2 f_1-x_4f_4 = 0$, so $x_2f_1 = x_4 f_4$, and so $x_2x_1t = x_4 f_4$, therefore $x_4^2t = x_4f_4$. Since $x_4$ is left regular, we have $f_4 = x_4t$.
We have proved that $\begin{bmatrix} f_1 & f_2 & f_3 & f_4 \end{bmatrix}^T$ is in the image of $M_4$,
and so the complex is exact at $S(-3)^4$.
%We have proved that $$\begin{bmatrix} f_1\\ f_2 \\ f_3 \\ f_4 \end{bmatrix} = \begin{bmatrix} x_1 \\ x_2 \\ x_3 \\ x_4 \\ \end{bmatrix} t,$$ and so the complex is exact at $S(-3)^4$.
We conclude that $S$ is Koszul.
%I write down a complex that is a potential minimal resolution of $\k_S$. I prove it is exact at the left two terms. I know by general theory that it is exact at the right two terms. Then by the previous proposition, i.e., because I know the Hilbert series of $S$, I get exactness in the middle. Details go in tomorrow. 
\end{proof}

Now we compute the quadratic dual of $S$. Let
$a_1, a_2, a_3, a_4$ be a dual basis to 
$x_1, x_2, x_3, x_4$. Then it is easy to see that $S^!$ 
has a presentation as
${\k \la a_1, a_2, a_3, a_4 \ra}$ modulo the relations:
%$$a_1^2, \quad a_2^2,\quad  a_3a_4,   \quad a_4a_3,$$
%$$a_1a_3+a_2a_4, \quad a_3a_2+a_4a_1, \quad a_3a_1+a_2a_3,$$
%$$a_1a_4+a_4a_2, \quad a_1a_2+a_3^2,  \quad  a_2a_1+a_4^2 .$$
\begin{eqnarray*}
&a_1^2,a_2^2,a_3a_4,a_4a_3,a_1a_3+a_2a_4,a_3a_2+a_4a_1,&\\
&a_3a_1+a_2a_3,a_1a_4+a_4a_2,a_1a_2+a_3^2,a_2a_1+a_4^2.&
\end{eqnarray*}

\begin{proposition}\label{Frobenius} The quadratic dual algebra,
$S^!$, is Frobenius.
\end{proposition}

\begin{proof} Order the generators as $a_4 < a_2 < a_3 < a_1$. Then it is straightforward to check that a Gr\"obner basis for $S^!$ in degree three is: $$\{a_4^2a_1,~~a_2a_4^2,~~ a_2 a_4 a_2-a_4^3,~~a_3^3+a_2a_4a_1,~~
    a_2a_3^2-a_4^2 a_2,~~a_4a_2a_3\},$$
and is $\{a_4^2 a_2 a_4, a_4^3a_2, a_4a_2a_4a_1\}$ and $\{a_4^5\}$
in degrees four and five, respectively.  We know the Hilbert series of $S^!$ is $(1+t)^4$, which along with the Gr\"obner basis helps to prove that a monomial basis for $S^!$ is
$$\{
1,
a_1,a_2,a_3,a_4,
a_2a_4,a_4a_1,a_2a_3,a_4a_2,a_3^2,a_4^2,
a_4^3,a_2a_4a_1,a_4^2a_2,a_4a_2a_4,
a_4^4
\}.$$
%\begin{align*}
%&S^!_0:  \quad 1 \\
%&S^!_1:  \quad a_1, \quad a_2,  \quad a_3,  \quad a_4 \\
%&S^!_2: \quad a_2a_4, \quad a_4a_1,  \quad a_2a_3,  \quad a_4a_2,  \quad a_3^2,  \quad a_4^2 \\
%&S^!_3: \quad a_4^3,  \quad a_2a_4a_1,  \quad a_4^2a_2,  \quad a_4a_2a_4 \\
%&S^!_4:  \quad a_4^4.
%\end{align*}

Now consider the pairing: %\textcolor{red}{(I think I prefer $\la -,- \ra$)}
$$\la -,- \ra : S^! \tsr S^!\to \k$$
given by $\la u, v \ra = \text{(the component of } uv \text{ in } (S^!)_4 = \k a_4^4$). To prove $S^!$ is Frobenius we show that this pairing is non-degenerate. We compute the values of the pairing using the ordered monomial basis above.

The computations are easy, albeit tedious.  We write matrices with the values of $\la -,- \ra$ on basis elements as entries. Note that only the restriction of the pairing to the subspaces $S^!_2 \tsr S^!_2$, $S^!_1 \tsr S^!_3$ and $S^!_3 \tsr S^!_1$ are relevant here.

On $S^!_2 \tsr S^!_2$, the matrix of the pairing is $I_6$, where $I_6$ is 
the $6 \times 6$ identity matrix.
On $S^!_1 \tsr S^!_3$, ordering rows by
$\{a_1, a_2, a_3, a_4\}$ and columns by
$\{a_4^3,a_2a_4a_1,a_4^2a_2,a_4a_2a_4\}$ the matrix of the pairing is 
$$
\begin{bmatrix} 
0 & 0 & -1 & 0 \\
0 & 0 & 0 & 1 \\
0 & -1 & 0 & 0 \\
1 & 0 & 0 & 0 \\
\end{bmatrix}. 
$$
On $S^!_3 \tsr S^!_1$, ordering rows by $\{a_4^3,
a_2a_4a_1,a_4^2a_1,a_4a_2a_4\}$ 
%{\color{red} $\{a_1, \; a_2,\;  a_3, \; a_4\}$  (should this be omitted??)}
and columns by $\{a_1,a_2,a_3,a_4\}$, the matrix of
the pairing is 
$$
\begin{bmatrix} 
0 & 0 & 0 & 1 \\
0 & 0 & -1 & 0 \\
-1 & 0 & 0 & 0 \\
0 & 1 & 0 & 0 \\
\end{bmatrix}. 
$$
We conclude that $\la -, - \ra$ is nondegenerate and so $S^!$ is Frobenius. 
%I write down the quadratic dual. I know its Hilbert series is $(1+t)^4$. I determine a graded monomial basis by Gr\"obner basis calculation. (Note to self: Try to find a nice basis, the one I chose originally may not be as nice as possible.) I show the pairing $$\la \_, \_ \ra : S^! \tsr S^! \to \k$$ given by $\la u, v \ra = \text{ component of } uv \text{ in } (S^!)_4$ is non-degenerate.
\end{proof}

\begin{theorem}\label{regular} The algebras $S$ and $T$ are Artin-Schelter regular of dimension 4.
\end{theorem}

\begin{proof} We know from Theorem \ref{Koszul} that $S$ has 
global dimension four, and has GK-dimension four. According to 
\cite[ Proposition 5.10]{Sm}, since $S^!$ is Frobenius we can 
conclude that $S$ is Gorenstein. Therefore $S$ is Artin-Schelter 
regular of dimension four. The same method can be used to prove 
that $T$ is Artin-Schelter regular of dimension four.
%
%Apply a result of Paul Smith.
\end{proof} 

\begin{remark}
    $S$ has only one central element in degree two up to scalar multiple, namely $x_1^2 + x_2^2 + x_3x_4 + x_4x_3$.
%(again the sum of the basis element we chose for $S^!_2$ see below).
Hence this algebra is not graded isomorphic to either the Sklyanin algebra or the algebra of the previous section. Note also that $S$ and $T$ are not isomorphic to each other, since it is straightforward to check that $T$ has no central (or normal) elements of degree two.\\
\end{remark}

Now we determine the group-graded identity components of $S$ and $T$.

\begin{proposition}
\label{invar-s}
The invariant subalgebra $S_e$ is the subalgebra of $S$ generated by the set $Y=\{x_1^2, x_2^2, x_3x_4,x_4 x_3\}$, and is isomorphic to $\mathbb{K}[z_1,\dots, z_4]$, the commutative polynomial ring in four variables.    
\end{proposition}

\begin{proof}
    The proof will be an application of \Cref{free modules}, similar to \Cref{invar-r}. Set $B$ as the subalgebra of $S$ generated by $Y$, $I$ as the ideal of $S$ generated by $Y$, $J$ as the ideal of $B$ generated by $Y$, and $C$ to be the polynomial ring $\k[z_1,\dots, z_4]$ in four variables. Denote the elements of $Y$ by $y_1,\dots, y_4$. It can be checked using the relations of $S$ that the elements of $Y$ pairwise commute, so that we have a ring homomorphism $\pi: C \to B$ given by $\pi(z_i) = y_i$.
    
    We can additionally check that $I=JS = SJ$ by showing that for each $i,j=1,\dots, 4,$ we have $x_i y_j = y_k x_i$ for some $k=1,\dots, 4$. Using the diamond lemma, one may check that the images of the following elements constitute a $\mathbb{K}$-basis for $S/I$: 
    %\begin{enumerate}
    %    \item In degree 0: 1;
    %    \item In degree 1: $x_1, x_2, x_3, x_4$; 
    %    \item In degree 2: $x_3^2$, $x_2 x_4$, $x_3 x_1$, $x_2 x_1$, $x_3 x_2$, $x_1 x_4$;
    %    \item In degree 3: $x_3^2 x_1, x_3^3, x_2 x_1 x_4, x_3x_2x_4$;
    %    \item In degree 4: $x_3^4$. 
    %\end{enumerate}
    $$\{1,
        x_1,x_2,x_3,x_4,
        x_3^2,x_2x_4,x_3x_1,x_2x_1,x_3x_2,x_1x_4,
        x_3^2x_1,x_3^3,x_2x_1x_4,x_3x_2x_4,
        x_3^4\}$$
    That is, the Hilbert series of $S/I$ is $(1+t)^4$. By Proposition \ref{Hilbert series} the Hilbert series of $S$ is $(1-t)^{-4}$, so we have $H_{S/I}(t)=(1-t^2)^4 H_{S}(t),$ and hence all assumptions of \Cref{free modules} are satisfied. It follows that $\pi: C \to B$ is an isomorphism, that $S$ is a free $B$-module of rank 16, and the elements in (1)-(5) above constitute a $B$-basis for $S$. It is clear that $B \subseteq S_e$; to complete the proof, just note that all non-unital elements of the $B$-basis of $S$ given above are in nontrivial $SD_{16}$-degree, and hence $S_e = B . 1 =B.$
\end{proof}

The proof of the next result is similar to that of \Cref{invar-s} and is omitted. It is interesting to note that the identity component $T_e$ is not a commutative polynomial ring. 

\begin{proposition}
\label{invar-t}
The invariant subalgebra $T_e$ is the subalgebra of $T$ generated by the set $Y=\{y_1=x_1^2, \;y_2= x_2^2,\; y_3=x_3x_4,\; y_4 = x_4 x_3\}$. The pairs $(y_1, y_2)$, $(y_1, y_3)$, $(y_2, y_4)$ and $(y_3, y_4)$ $(-1)$-skew commute, whereas the pairs $(y_1, y_4)$ and $(y_2, y_3)$ commute. 
Hence $T_e$ is a skew polynomial ring and therefore AS regular. 
% {\color{red}This is a complete defining set of relations for $T_e$. and is isomorphic to $\mathbb{K}[z_1,\dots, z_4]$,  the polynomial ring in four variables. (This is wrong. See Proposition 4.29 in ASReg1006 file for the correct statement.)}  {\color{brown}ek: The statement is now the same as Proposition 4.29.}
\end{proposition}

\subsection{Algebraic and homological properties of \texorpdfstring{$R,S,$}{R,S} and \texorpdfstring{$T$}{T}}
 In the previous sections we have shown that $R,S,$ and $T$ are AS regular algebras of dimension four and that each can be homogeneously graded by a group of order 16 so that the identity component, the fixed ring under the action of the dual, is also AS regular.  In this section we prove several more algebraic and homological properties of these algebras.

% \begin{proposition} Using the notation $x_1 = Y_1, x_2 = Y_2, x_3 = Z_1,$ and $x_4 = Z_2$ in the double Ore extension presentation of $R$ the following elements of 
% $R$ are normal elements of $R$:
% $$n_1 = x_2x_1$$
% $$n_2 = x_3x_4$$
% $$n_3 = x_1^2-x_2^2$$
% $$n_4 = x_3^2-x_4^2$$
% $$n_5 = x_1^2+x_2^2$$
% $$n_6 = x_3^2+x_4^2.$$
% The center of $R$ is generated by 
% $$w_1 = n_3 = x_1^2 - x_2^2$$
% $$w_2 = n_4= x_3^2 - x_4^2$$
% $$w_3 = n_1^2 = x_2x_1x_2x_1 = -x_2^2x_1^2$$
% $$w_4 = n_2^2 = - x_3^2x_4^2$$
% $$w_5 = n_1n_2= - x_1x_2x_3x_4$$
% $$w_6 = n_1n_5n_6$$
% $$w_7 = n_2n_5n_6.$$

% \end{proposition}
\begin{lemma}\label{central regular sequences}
The following sequences are central, regular sequences in $R$, $S$, $T$, respectively:
% {\color{blue} (Note: The presentation of $R$ is the double Ore presentation (aka the one in the document we sent Colin.))}
\begin{itemize}
\item[(1)] $x_2^2-x_1^2$, $x_4^2-x_3^2$, $x_1^2x_2^2$, $x_3^2x_4^2$; %{\color{blue} IMPORTANT: $x_1^4$, $x_3^4$ are not central; need to fix the proof below.}  {\color{brown}  Ellen notes the elements $x_3^4$ and $x_4^4$ are central in the presentation that is currently in the geometry paper}
\item[(2)] $x_4 x_3+x_1^2 + x_2^2 + x_3 x_4$, $x_1^2x_4x_3+x_3x_1^2x_4$, $x_2^2 x_4 x_3+x_2^2 x_1^2+x_3 x_2^2 x_4+x_3^2 x_2 x_1$, $x_3^4$;
\item[(3)] $(x_4 x_3)^2+x_1^4+x_2^4+(x_3 x_4)^2$,
$x_3^4$, 
$x_1^4 (x_4 x_3)^2 + x_3 x_1^4 x_4 x_3 x_4$,
$x_2^4 (x_4 x_3)^2 +x_2^4 x_1^4+x_3 x_2^4 x_4 x_3 x_4+x_3^2 x_2^3 x_1^3$.
\end{itemize}
\end{lemma}

\begin{proof}
The method of proof is the same for (1), (2) and (3). 
% {\color{blue} Pete: Cases (2) and (3) are more complicated, but the proof method is the same. I'll include all the details and we can decide later what to cut.) } %It is straightforward to check that the stated elements are central in $R$, $S$, $T$, respectively.\\
% ~\\

For (1), order monomials in the free algebra $\k \la x_1, x_2, x_3, x_4 \ra$ according to $x_1 < x_2 < x_3 < x_4$ and left-lexicographic degree order. Using Bergman's diamond lemma, it is easy to check that the defining ideal of $R$ has a quadratic Gr\"obner basis, and that $R$ has a monomial basis given by $x_1^ix_2^jx_3^kx_4^l$, $i, j, k, l \geq 0$. This makes it straightforward to check that the elements given in (1) are indeed central.   Let $$R' = R/\la x_2^2-x_1^2, x_4^2-x_3^2, x_1^2x_2^2, x_3^2x_4^2 \ra.$$
Another application of the diamond lemma shows that the defining ideal of $R'$ has a finite Gr\"obner basis whose leading terms are: $$x_2 x_1, x_3 x_1, x_3 x_2, x_4 x_1, x_4 x_2, x_4 x_3, x_2^2, x_4^2, x_1^4, x_3^4.$$ This makes it clear that $R'$ has a monomial basis given by: $$x_1^i x_2^j x_3^k x_4^l, \quad 0 \leq i, k \leq 3, \quad 0 \leq j, l \leq 1.$$
Hence $$H_{R'}(t) = \left(\dfrac{1-t^2}{1-t}\right)^2 \left(\dfrac{1-t^4}{1-t}\right)^2  = (1-t^2)^2(1-t^4)^2H_R(t).$$ 
We conclude by Proposition \ref{normal regular sequences via Hilbert series} 
that the sequence in (1) is regular. 

For (2), order monomials in the free algebra $\k \la x_1, x_2, x_3, x_4 \ra$ according to $x_3 < x_2 < x_1 < x_4$ and left-lexicographic degree order. Recall,  by  \Cref{G-basis} and \Cref{monomial basis}, that the defining ideal of $S$ has a finite Gr\"obner basis, and that $S$ has a monomial basis given by $x_3^ix_2^jx_1^k(x_4x_3)^lx_4^m$, $i, j, k, l \geq 0$, $0 \leq m \leq 1.$   This makes it straightforward to check that the elements given in (2) are indeed central.   Let $$S' = S/\la x_4 x_3+x_1^2 + x_2^2 + x_3 x_4, x_1^2x_4x_3+x_3x_1^2x_4, x_2^2 x_4 x_3+x_2^2 x_1^2+x_3 x_2^2 x_4+x_3^2 x_2 x_1, x_3^4 \ra.$$ Another application of the diamond lemma shows that the defining ideal of $S'$ has a finite Gr\"obner basis whose leading terms are: $$x_1 x_2, x_4^2, x_1 x_3, x_2 x_3, x_4 x_2, x_4 x_1, x_4x_3, x_1^4, x_2^4, x_3^4.$$ This makes it clear that $S'$ has a monomial basis given by: $$x_3^i x_2^j x_1^k x_4^l, \quad 0 \leq i, j, k \leq 3, \quad 0 \leq l \leq 1.$$

Hence $$H_{S'}(t) = \left(\dfrac{1-t^2}{1-t}\right) \left(\dfrac{1-t^4}{1-t}\right)^3  = (1-t^2)(1-t^4)^3 H_S(t).$$ 
%{\color{brown} Shouldn't this be $ (1-t^2)^2(1-t^4)^2 \cdot H_D(t)$?} 

We conclude by Proposition \ref{normal regular sequences via Hilbert series} 
that the sequence in (2) is regular. 

For (3), order monomials in the free algebra $\k \la x_1, x_2, x_3, x_4 \ra$ according to $x_3 < x_2 < x_1 < x_4$ and left-lexicographic degree order. Using Bergman's diamond lemma, it is easy to check that the defining ideal of $T$ has a finite Gr\"obner basis, and that $T$ has a monomial basis given by $x_3^ix_2^jx_1^k(x_4x_3)^lx_4^m$, $i, j, k, l \geq 0$, $0 \leq m \leq 1.$  
% {\color{blue} (refer to a previous result in the paper; note that this hasn't been explicitly written down anywhere, but it is implicit in the claims made about the algebra $T$ and Pete has checked it by hand)} 
This makes it straightforward to check that the elements given in (3) are indeed central.   Let $T' = T/\la z_1,z_2, z_3, z_4 \ra$
where 
$$z_1 = (x_4 x_3)^2+x_1^4+x_2^4+(x_3 x_4)^2,\quad
z_2= x_3^4, \quad
z_3= x_1^4 (x_4 x_3)^2 + x_3 x_1^4 x_4 x_3 x_4,$$
$$z_4 =x_2^4 (x_4 x_3)^2 +x_2^4 x_1^4+x_3 x_2^4 x_4 x_3 x_4+x_3^2 x_2^3 x_1^3.$$
Another application of the diamond lemma shows that the defining ideal of $T'$ has a finite Gr\"obner basis whose leading terms are: $$x_1 x_2, x_4^2, x_1 x_3, x_2 x_3, x_4 x_2, x_4 x_1, x_4x_3x_1, x_4x_3x_2, x_4x_3^2, (x_4x_3)^2, x_3^4, x_1^8, x_2^8.$$ 
% {\color{blue} (Pete notes that this is what M2 says, but he hasn't checked the diamond lemma calculations by hand.)} 
This makes it clear that $T'$ has a monomial basis given by: $$x_3^i x_2^j x_1^k(x_4x_3)^l x_4^m, \quad 0 \leq i \leq 3, \quad 0 \leq j, k \leq 7, \quad 0 \leq l, m \leq 1.$$

Hence $$H_{T'}(t) = (1+t^2) \left(\dfrac{1-t^2}{1-t}\right) \left(\dfrac{1-t^4}{1-t}\right) \left(\dfrac{1-t^8}{1-t}\right)^2  = (1-t^4)^2(1-t^8)^2H_T(t).$$ 
%{\color{brown} Shouldn't this be $ (1-t^2)^2(1-t^4)^2 \cdot H_D(t)$?} 

We conclude by Proposition \ref{normal regular sequences via Hilbert series} 
that the sequence in (3) is regular.  
\end{proof}

% {\color{brown} EK: I checked by hand that all the elements in the Lemma  are central -- with the exception of the last two elements in case (3) (which I gave up doing by hand) -- these computations are not too hard to do by hand.} 

% Ellen: For $R$ are $x_2^2-x_1^2, x_4^2 - x_3^2, x_2^2x_1^2, x_3^2x_4^2$  algebraic independent? {\color{blue} (Pete: Yes, these elements are algebraically independent, and the same is true for the elements in (2) and (3) above. This follows from Lemma 0.1 above and  Proposition 4.14 (2) in the April 2024 folder.)}  Note
%that $$x_1^2x_2^2 = x_1x_1 x_2 x_2 = - x_1 x_2x_1 x_2= x_2x_1x_1 x_2 =  -x_2x_1x_2 x_1 = x_2x_2 x_1x_1=x_2^2x_1^2$$
%$$(x_2^2-x_1^2)^2 =x_2^4 - 2 x_1^2x_2^2 + x_1^4$$
%seems  like this should help.  (similarly with $x_3^2x_4^2$). Cases for $S$ and $T$ look more complicated to me.
\begin{corollary} \label{consequences of PI}
The algebras $R$, $S$ and $T$ are all noetherian PI algebras of finite global dimension, and hence by \cite{SZ} their centers are noetherian integrally closed domains, and
they are:
\begin{itemize}
\item[(1)] Auslander-Gorenstein,
\item[(2)] Auslander-regular,
\item[(3)] GK-Cohen-Macaulay ($j(M) + \operatorname{GKdim}(M) =  \operatorname{GKdim}(R)$ for all finitely generated modules, where
$j(M) = \min\{j: \Ext_R^j(M,R) \neq 0\}$),
\item[(4)] Macaulay (in the sense of \cite{SZ}: $j(M) + \operatorname{Kdim}(M) = \operatorname{Kdim}(R)$ for all finitely generated modules $M$, where $\operatorname{Kdim}$ denotes Krull dimension),
\item[(5)] maximal orders in their quotient division rings.

\end{itemize}
\end{corollary}

\begin{proof}
Combining \Cref{central regular sequences} and \Cref{central regular sequence implies module-finite} (1) we see that each of $R$, $S$ and $T$ is a finitely generated module over its center. Hence these algebras are PI. The rest of the properties follow from work of Stafford-Zhang \cite{SZ}.
\end{proof}

\section{Geometry associated to \texorpdfstring{$R$}{R}, \texorpdfstring{$S$}{S} and \texorpdfstring{$T$}{T}}
\label{sect-geom}

In this section we study the geometry of points and lines for the algebras $R, S,$ and $T$.
For the computations, we fix a primitive eighth root
of unity $\zeta \in \k$, and set $\i = \zeta^2$.
 As noted in the introduction, quadratic AS regular algebras of dimension four with twenty points and a one-dimensional line scheme are called ``generic"; we will show that $S$ and $T$ are generic.

\subsection{Computing point and line schemes}
 In this subsection we describe how the point schemes and line schemes were computed.
 %For the computations, we fix a primitive fourth root of unity $\i \in \k$,
 %and fix a primitive eighth root of unity $\zeta \in \k$ such that $\zeta^2 = \i$.
%Here we study the geometry of points and lines for our algebras. 
In order to prove that certain schemes are reduced, we will use the following well-known lemma. 

\begin{lemma}\label{checking reduced on affine cover}\cite[Lemma 28.3.2(2)]{stacks}
Let $X$ be a scheme and assume that there is an affine open cover $\{U_i = \Spec A_i\}_{i \in I}$ such that $A_i$ is a reduced ring for all $i \in I$. Then $X$ is reduced.
\end{lemma}

First we discuss the computation of point schemes. We begin by recalling some definitions. Let $V$ be a finite-dimensional $\k$-vector space. Let $A = T(V)/I$ be a finitely presented quadratic $\k$-algebra.

\begin{definition}\label{def of point module}
A graded right $A$-module $M$ is a \emph{point module} if $M$ 
is cyclic, generated in degree $0$, and the Hilbert series of 
$M$ is $H_M(t) = (1-t)^{-1}$.  The isomorphism class of a point
module over $A$ is called a \emph{point of $A$}.
\end{definition} 

Following \cite{ATV1}, we now explain a condition that implies that
isomorphism classes of point modules are parametrized by a projective scheme. 
Let $\G$ be the scheme of zeros in $\P(V^*) \times \P(V^*)$ of the 
multilinearizations of the defining relations of $A$. Let $E$ be the
scheme-theoretic image of the canonical projection $\pi_1: \G \to \P(V^*)$.
If $\G$ is the graph of a scheme automorphism $\t: E \to E$, then it is well 
known that the set of closed points of $E$ is in bijective correspondence with
the set of isomorphism classes of point modules of $A$, and hence the points
of $A$.  In this case, the scheme $E$ is called the \emph{point scheme of $A$}.

If $A$ is a quadratic AS regular algebra of global dimension two or three, then
$\G$ is the graph of an automorphism. A theorem of Shelton-Vancliff,
\cite[Theorem 1.4]{SV1}, states that if $A$ is a noetherian,
Auslander-regular, GK-Cohen Macaulay algebra of global dimension four
with Hilbert series $(1-t)^{-4}$, then $\G$ is the graph of a scheme
automorphism $\tau: E \to E$.
%It is unknown if every quadratic AS regular algebra satisfies the
%condition that $\G$ is the graph of an automorphism.

Let $A$ be any of the algebras $R$, $S$ or $T$. We have proved that 
$A$ has global dimension four and Hilbert series $(1-t)^{-4}$. 
Furthermore, by Corollary \ref{consequences of PI} the hypotheses of 
\cite[Theorem 1.4]{SV1} are satisfied. Hence $\G$ is the graph of a 
scheme automorphism $\t: E \to E$ and we may speak of the point 
scheme, $E$.  We will also denote the point scheme of $A$ by 
${\mathcal P}(A)$.

We follow the method outlined in \cite{ATV1}. To compute $E$, let
$V = \bigoplus_{i=1}^4 \k x_i$ be the vector space of degree one 
elements of $A$. Let $\{e_i\}_{1 \leq i \leq 4} \subset V^*$ be the 
dual basis to $\{x_i\}_{1 \leq i \leq 4}$. We identify $\P(V^*)$ with 
the projective space $\P^3$ using the ordered basis
$\{e_i\}_{1 \leq i \leq 4}$.
%{\color{red} (PG: I think I'm of the opinion that we write the space of degree one elements of $A$ as $V^*$ because then elements of that space are functions on $\k$, and the point scheme makes sense as a scheme of zeroes in $\mathbb{P}(V)$ of homogeneous polynomials in the $x_i$.)} 
We express the defining relations of $A$ as a product of the form
$M {\bf x} = {\bf 0}$, where $M$ is a $6 \times 4$ matrix with 
entries in $V$, and ${\bf x}$ is the column vector
${\bf x} = (x_1, x_2, x_3, x_4)^T$.
We consider the fifteen $4 \times 4$ minors of $M$ as homogeneous 
polynomials in the symmetric algebra on $V$. Identifying $V$ with 
$V^{**}$ in the usual way, the point scheme of $A$ is the zero locus 
in $\mathbb{P}(V^*)$ of the $4 \times 4$ minors of $M$. In general, 
if $X$ is a set of polynomials, we will write $\mathcal{V}(X)$ for 
the zero set of $X$. 

The third author has written a function contained in the {\it Macaulay2}
\cite{Mac2} package 
{\it AssociativeAlgebras} called {\tt pointScheme} to aid in the computation of 
the point scheme; the output of this function depends on the presentation of the 
algebra. To compute the automorphism $\t: E \to E$, let $p \in E$ be a closed 
point. The matrix, $M(p)$, obtained by evaluating $M$ at $p$ (note 
that this matrix is well-defined up to a scalar) has rank three and
$\t(p)$ is the kernel of $M(p)$.
% {\color{brown} In this section we will compute the line schemes and line modules of $R,S,$ and $T$. As noted in the introduction, quadratic AS regular algebras of dimension 4 with twenty points and a one-dimensional line scheme are called ``generic"; we will show that $S$ and $T$ (that were shown in the previous section to have twenty points) are generic.

Next we discuss the construction of the line scheme of $A$.
\begin{definition}\label{def of line module}
A graded right $A$-module $M$ is a \emph{line module} if $M$ is cyclic,
generated in degree $0$ and the Hilbert series of $M$ is
$H_M(t) = (1-t)^{-2}$.  As before, the isomorphism class of a line
module over $A$ is called a \emph{line of $A$}.
\end{definition}

%{\color{red} (PG: I think we should add some comments about why the line scheme parametrizes line modules for our algebras, $R$, $S$, $T$. See a theorem of Shelton-Vancliff. Again, the necessary properties are the ones following from the PI condition.)} 

In \cite[Remark 2.10]{SVline1}, it is shown that if $A$ is a quadratic, 
noetherian, Auslander-regular algebra of global dimension four which satisfies 
the Cohen-Macaulay property with respect to GK-dimension, then there is a 
projective scheme, the \emph{line scheme}, that parametrizes isomorphism classes 
of line modules, and hence lines of $A$. As noted above, the algebras $R$, $S$ 
and $T$ satisfy these properties.

We determine line schemes using the method introduced in \cite{SVline2}. For a graded  algebra
$A$ with four generators $\{x_1, x_2, x_3, x_4\}$ in degree one  and six quadratic relations $\{r_j: j=1, \dots,6\}$ one forms the Koszul dual $A^!$ using the dual basis 
$\{z_1, z_2, z_3, z_4 \}$ of the dual of the vector space $A_1 = \sum \k x_i$ and the ten defining relations $\{f_i: i = 1, \dots, 10 \}$ of $A^!$ so that $f_i(r_j) =0$ for all $i,j$. The defining relations of $A^!$ can be expressed as the matrix equation $\hat{M}{\bf z} = {\bf 0}$ where $\hat{M}$ is a $10 \times 4$ matrix with entries that are linear forms in the $\{z_i\}$ and ${\bf z}$ is the column vector ${\bf z}= (z_1, z_2, z_3, z_4)^T$.

Consider the polynomial ring $\k[u_1, \dots, u_4, v_1, \dots v_4]$. A $10 \times 8$ matrix $\mathcal{M} = [\hat{M}_u \hat{M}_v]$ is obtained by horizontally concatenating the $10 \times 4$ matrices $\hat{M}_u$ and $\hat{M}_v$, where $\hat{M}_u$ is formed from $\hat{M}$ by replacing each variable $z_i$ with the variable $u_i$, and $\hat{M}_v$ is formed from $\hat{M}$ by replacing each variable  $z_i$ with the variable $v_i$.

The forty-five $8 \times 8$ minors of $\mathcal{M}$ are each linear combinations of polynomials of the
form $N_{ij} = u_iv_j- u_jv_i$, where $1 \leq i < j \leq 4$, giving forty-five quartic polynomials in the six variables $N_{12}, N_{13}, N_{14}, N_{23}, N_{24}, N_{34}$. Next, to these polynomials apply the orthogonality isomorphism described in \cite[Section 1]{SVline2}:
$$\begin{array}{ccc}
N_{12} \mapsto M_{34} & N_{13} \mapsto -M_{24} & N_{14} \mapsto M_{23} \\
N_{23} \mapsto M_{14} & N_{24} \mapsto -M_{13} & N_{34} \mapsto M_{12}.
\end{array}$$
%{\color{red} Are our $X_{ij} = M_{ij}$??? YES}
 This process produces forty-five polynomials in the Pl\"{u}cker coordinates $M_{12}, M_{13},$ $ M_{14}, M_{23}, M_{24}, M_{34}$ on $\mathbb{P}^5$.  The line scheme $\mathcal{L}(A)$ is the scheme of zeros in $\mathbb{P}^5$ of these forty-five quartic polynomials along with the Pl\"{u}cker polynomial $P = M_{12}M_{34} - M_{13} M_{24} + M_{14}M_{23}$. 
 %{\color{red} (Important note for selves: This description of $\mathcal{L}(A)$ does give the scheme denoted $\mathcal{L}$ in the Shelton-Vancliff papers. However, duly note that we are using $V$, whereas Shelton-Vancliff use $V^*$. Moreover, if we want to compute a line module for some point $\ell \in \mathcal{L}(A)$ (more precisely, $\ell$ realized as a line in $\P(V^*)$), that line module is $A/\ell^{\perp}A$.)}

For each of the algebras, $R$, $S$ and $T$, we will determine the set of closed points and the decompositions of their line schemes into irreducible components. %{\color{red} (PG: Do we want to try to prove that the line schemes for $R$, $S$, $T$ are reduced schemes? I believe we can mimic the proofs of Lemma 3.2 and Theorem 3.3 of Chandler-Vancliff (J of A 2015) and show that the line schemes of $S$ and $T$ are reduced. I'm less sure how to make the argument for $R$.)} 
The third author has written a {\it Macaulay2} function {\tt lineSchemeFourDim} %{\color{red} (how should we format Macaulay2 functions???)} 
that takes as input a presentation of a graded algebra $A$ with four degree one generators and six quadratic relations, and outputs the ideal of the polynomial ring in the variables $M_{ij}$ generated by the forty-five quartic polynomials described above and the Pl\"ucker polynomial. This output depends on the presentation. 
%We will use the notation $M=\k[M_{12},M_{13},M_{14},M_{23},M_{24},M_{34}]$ and write $I$ for the ideal produced by the Macaulay2 function ``lineSchemeFourDim".

 Next we describe the lines in $\mathbb{P}^3$ that are parameterized by the line scheme $\mathcal{L}(A)$.  We write ${\mathbb G} = {\mathbb G}(1, \P^3)$ for the Grassmannian scheme parametrizing lines in $\P(V^*) \cong \P^3$. The line through two distinct points $p = [a_1:a_2:a_3:a_4]$ and $q = [b_1:b_2:b_3:b_4]$ in $\mathbb{P}^3$ will be denoted by $\ell(p,q)$. This line is the projectivization of the row space of the (rank two) $2 \times 4$ matrix
$\left[\begin{smallmatrix}
a_1 & a_2 & a_3 & a_4\\
b_1 & b_2 & b_3 & b_4
\end{smallmatrix}\right].$
 We will abuse notation and also denote the line through $p$ and $q$ by this matrix. 

 The evaluation of a Pl\"{u}cker coordinate $M_{ij}$ on $\ell = \ell(p,q)$ is defined to be $a_ib_j-a_jb_i$.  The map $\pi: {\mathbb G} \to \P^5$, $\ell \mapsto [M_{12}(\ell):M_{13}(\ell):M_{14}(\ell):M_{23}(\ell):M_{24}(\ell):M_{34}(\ell)]$ is well known to give an isomorphism from ${\mathbb G}$ to ${\mathcal V}(P)$, where $P$ is the Pl\"ucker polynomial. 

Given a line $\ell \in {\mathbb G}$, define a right $A$-module by $M(\ell) = A/\ell^{\perp} A$. Then $M(\ell)$ is a line module of $A$ precisely when $\ell \in  \pi^{-1}(\mathcal{L}(A))$.

 In the following subsections we will compute the irreducible components of the line schemes of $R$, $S$ and $T$. If $C$ is such a component, we will determine the set of lines of $\P^3$ in the preimage $\pi^{-1}(C)$. We establish some notation to describe these subschemes of $\mathbb{G}$. If $F = F(x_1, x_2, x_3, x_4)$ is a linear polynomial, so that $\mathcal{V}(F)$ is a plane in $\P^3$, then $\mathcal{H}(F)$ will denote the subscheme of $\mathbb{G}$ whose closed points are all of the lines contained in the plane $\mathcal{V}(F)$. If $Y$ and $Z$ are disjoint subschemes of $\P^3$, we let $\mathcal{J}(X,Y)$ denote the subscheme of $\mathbb{G}$ whose closed points are the lines joining points of $Y$ and $Z$. If $p \in \P^3$ is a point and $H \subset \P^3$ is a plane, the pencil of lines through $p$ in the plane $H$ will be denoted by $\Sigma_{p,H}$.

Finally, we will study some incidence geometry. Let $A$ be any of the algebras $R$, $S$ or $T$. For each closed point $p$ in the point scheme of $A$, let $\mathcal{L}(p)$ denote the subscheme of $\mathbb{G}$ whose closed points consist of the lines in $\pi^{-1}(\mathcal{L}(A))$ that contain $p$. Also, for any line $\ell \subset \P^3$, let $\mathcal{P}(\ell)$ denote the subscheme of $\P^3$ whose closed points consist of all $p \in \mathcal{P}(A)$ that lie on $\ell$. We determine:
\begin{itemize} 
\item[(1)] for each point $p \in \mathcal{P}(A)$, the set of closed points in $\mathcal{L}(p)$; 
\item[(2)] for each line $\ell \in \pi^{-1}(\mathcal{L}(A))$, the set of closed points of $\mathcal{P}(\ell)$.
%\item[(3)] determine all surjective graded maps $L \to M$, where $L$ is a line module and $M$ is a point module;
\end{itemize}
%{\color{red} (PG: Perhaps we won't do item (3)? But this is the connection of the geometry back to the algebra $A$, and we might see something....)}

%{\color{red} The correspondence between points in $\P(V^*)$ and point modules is as follows. For $p = [a:b:c:d] \in E \subset \P^3 \cong P(V^*)$, consider the orthogonal space $p^{\perp} = \{v \in V : v(p) = 0\}$. Then the point module of $A$ corresponding to $p$ is given by: $$\mathcal{P}(p) = A/p^{\perp}A.$$ 

\subsection{The point and line schemes of \texorpdfstring{$R$}{R}}
In this subsection we study the point and line schemes of the algebra $R$.

\begin{notation}\label{notation for geometry of R}
    The geometry in $\P(V^*)$ associated to the algebra $R$ can be succinctly described in terms of four hyperplanes. Let $$H_1 = \mathcal{V}(x_1-x_2), \quad H_2 = \mathcal{V}(x_1+x_2), \quad H_3 = \mathcal{V}(x_3-x_4), \quad H_4 = \mathcal{V}(x_3+x_4).$$
    These four hyperplanes pairwise intersect in six lines. For $1 \leq i < j \leq 4$, let $E_{i,j}$ denote the line $H_i \cap H_j$. It is easy to show that
\begin{eqnarray*}
&&E_{1,2} = \mathcal{V}(x_1,x_2) = \{[0:0:a:b] : [a:b] \in \P^1\}, \\
&&E_{3,4} = \mathcal{V}(x_3,x_4) = \{[a:b:0:0] : [a:b] \in \P^1\}, \\
&&E_{1,3} = \mathcal{V}(x_1-x_2,x_3-x_4) = \{[a:a:b:b] : [a:b] \in \P^1\}, \\
&&E_{2,4} = \mathcal{V}(x_1+x_2,x_3+x_4) = \{[a:-a:b:-b] : [a:b] \in \P^1\}, \\
&&E_{2,3} = \mathcal{V}(x_1+x_2,x_3-x_4) = \{[a:-a:b:b] : [a:b] \in \P^1\}, \\
&&E_{1,4} = \mathcal{V}(x_1-x_2,x_3+x_4) = \{[a:a:b:-b] : [a:b] \in \P^1\}. \\ 
\end{eqnarray*}
The nonempty triple intersections of these hyperplanes determine four 
points.  Let 
$$p_1 = [1:-1:0:0], \quad p_2 = [1:1:0:0], \quad
  p_3 = [0:0:1:-1], \quad p_4 = [0:0:1:1].$$
Indeed, if $\{i,j,k,l\} = \{1,2,3,4\}$, then $H_i \cap H_j \cap H_k = \{p_l\}$.
%\kv{Changed $H_l$ to $H_k$, feel free to change back if I misunderstood something.}
\end{notation}
The proof of the following theorem is a computation using the aforementioned {\it Macaulay2} functions, but the computations have also been checked by hand. We omit most of the details in the proof. %({\color{red} question: should we include some details???})

\begin{theorem} \label{point schemes, automorphisms of R}
%\begin{itemize}
%\item[(1)] 
Using the presentation for $R$ with generators $x_1, x_2, x_3, x_4$ and relations: $x_1x_2+x_2x_1$, $x_1x_3+x_4x_2$, $x_1x_4-x_3x_2$, $x_2x_3-x_4x_1$, $x_2x_4+x_3x_1$, $x_3x_4+x_4x_3$, the irreducible components of the point scheme for $R$ are the six lines $E_{i,j}$, $1 \leq i < j \leq 4$, the pairwise intersections of the hyperplanes $H_k$, $1 \leq k \leq 4$.
The automorphism $\tau: \mathcal{P}(R) \to \mathcal{P}(R)$ is given by:
\begin{eqnarray*}
&&\tau: E_{1,2} \rightarrow E_{1,2}: \tau([0:0:a:b]) = [0:0:a:-b], \\
&&\tau: E_{3,4} \rightarrow E_{3,4}: \tau([a:b:0:0]) = [a:-b:0:0], \\
&&\tau: E_{1,3} \rightarrow E_{2,4}: \tau([a:a:b:b]) = [a:-a:b:-b], \\
&&\tau: E_{2,4} \rightarrow E_{1,3}: \tau([a:-a:b:-b]) = [a:a:b:b], \\
&&\tau: E_{2,3} \rightarrow E_{1,4}: \tau([a:-a:b:b]) = [a:a:-b:b], \\
&&\tau: E_{1,4} \rightarrow E_{2,3}: \tau([a:a:b:-b]) = [a:-a:-b:-b].
\end{eqnarray*}
\noindent The scheme ${\mathcal P}(R)$ is equidimensional of dimension one and is a reduced scheme. 
\end{theorem}
\begin{proof}
With the stated presentation, the matrix given by factoring the relations is:
$$M := \begin{bmatrix} x_2 & x_1 & 0 & 0 \\
                      0 & x_4 & x_1 & 0 \\
                      0 & -x_3 & 0 & x_1 \\
                      -x_4 & 0 & x_2 & 0 \\
                      x_3 & 0 & 0 & x_2 \\
                      0 & 0 & x_4 & x_3 \\
      \end{bmatrix}.
$$ The {\it Macaulay2} function {\tt pointScheme} outputs the fifteen $4 \times 4$ minors of $M$ as an ideal $I$ of the polynomial ring $\Q[x_1, x_2, x_3, x_4]$. Then we use the {\it Macaulay2} function {\tt minimalPrimes} on the ideal $I$ to compute the irreducible components of the zero locus of $I$. We find that these components are precisely the six lines $E_{i,j}$, $1 \leq i < j \leq 4$. It is then clear that $\mathcal{P}(R)$ is equidimensional of dimension one. To prove that the scheme ${\mathcal P}(R)$ is reduced, one can check that the conditions of Lemma \ref{checking reduced on affine cover} hold for the standard affine open cover of $\P^3$, restricted to ${\mathcal P}(R)$. 
%The rest of the proofs for (1) and (3) are left to the reader. 
\end{proof}

 Next we describe the line scheme of $R$.  
%(Here we put in the decompositions of the line schemes in terms of irreducible components. Describe components as smooth conics, or $\mathbb{P}^2$'s, or smooth quadrics ($\mathbb{P}^1 \times \mathbb{P}^1$'s). Need to prove that line schemes for $S$ and $T$ are reduced and dimension 1.)

\begin{theorem} \label{line scheme of R}
Using the presentation for $R$ with generators $x_1, x_2, x_3, x_4$ and 
relations: $x_1x_2+x_2x_1$, $x_1x_3+x_4x_2$, $x_1x_4-x_3x_2$, $x_2x_3-x_4x_1$, 
$x_2x_4+x_3x_1$, $x_3x_4+x_4x_3$, the irreducible components of the line scheme 
for $R$ are the following:
%$$\begin{aligned}
% & C_1 = \V(M_{12},M_{13}-M_{23},M_{14}-M_{24}), \quad
%   C_2 = \V(M_{12},M_{13}+M_{23},M_{14}+M_{24}), \\
% & C_3 = \V(M_{34},M_{13}-M_{14},M_{23}-M_{24}), \quad
%   C_4 = \V(M_{34},M_{13}+M_{14},M_{23}+M_{24}),
%\end{aligned}$$
%$$\begin{aligned}
%& C_5 = \V(M_{12},M_{34},M_{14}M_{23}-M_{13}M_{24}), \\
%& C_6 = \V(M_{14}-M_{23},M_{13}-M_{24},M_{23}^2+M_{12}M_{34}-M_{24}^2), \\
%& C_7 = \V(M_{14}+M_{23},M_{13}+M_{24},M_{23}^2-M_{12}M_{34}-M_{24}^2). \\
%\end{aligned}$$
$$\begin{array}{ll}
C_1 = \V(M_{12},M_{13}-M_{23},M_{14}-M_{24}), &
C_2 = \V(M_{12},M_{13}+M_{23},M_{14}+M_{24}), \\
C_3 = \V(M_{34},M_{13}-M_{14},M_{23}-M_{24}), &
C_4 = \V(M_{34},M_{13}+M_{14},M_{23}+M_{24}), \\
C_5 = \V(M_{12},M_{34},P), &
C_6 = \V(M_{14}-M_{23},M_{13}-M_{24},P), \\
C_7 = \V(M_{14}+M_{23},M_{13}+M_{24},P),
\end{array}$$
where $P$ denotes the Pl\"ucker polynomial.
Components $C_1, C_2, C_3, C_4$ are isomorphic to $\P^2$; components $C_5, C_6, C_7$ are isomorphic to $\P^1 \times \P^1$. The scheme $\mathcal{L}(R)$ is equidimensional of dimension two and is a reduced scheme. %{\color{red} (PG: Is it actually reduced?? Needs checking!)}
\end{theorem}

\begin{proof}
   Let $B = \k[M_{i,j}: 1 \leq i < j \leq 4]$ be the polynomial ring in the Pl\"ucker coordinates. A computation using the aforementioned {\it Macaulay2} functions shows that the minimal primes of the ideal $I$ of $B$ determining the line scheme $\mathcal{L}(R)$ give the stated components $C_i$, $1 \leq i \leq 7$. It is then clear that $C_1$, $C_2$, $C_3$, and $C_4$ are isomorphic to $\P^2$. For $C_5$, $C_6$, and $C_7$, note that modulo the linear relations, the Pl\"ucker polynomial is smooth, which verifies that these components are isomorphic to $\P^1 \times \P^1$. The statement about the dimension of $\mathcal{L}(R)$ follows immediately.
   
   To prove that $\mathcal{L}(R)$ is reduced, let $N = B/I$. Let $U_{i,j} = D_{+}(M_{i, j}) \cong \Spec N_{(M_{i,j})}$, where $N_{(M_{i,j})}$ is the subalgebra of the localization $N_{M_{i,j}}$ of elements of degree $0$. Then we may consider the $U_{i,j}$ as an affine cover of the scheme $\mathcal{L}(R)$. The defining ideal of $N_{M_{i,j}}$ is obtained by dehomogenizing the ideal $I$ by setting $M_{i,j} = 1$. One can check, using {\it Macaulay2}, that for each $i, j$ this ideal is a radical ideal. It follows that $N_{(M_{i,j})}$ is a reduced ring. By Lemma \ref{checking reduced on affine cover}, $\mathcal{L}(R)$ is reduced.
\end{proof}

\begin{theorem}\label{lines in P^3 for R}
The preimages of the components of the line scheme of $R$ under
$\pi: {\mathbb G} \to \P^5$ satisfy:
$\pi^{-1}(C_i) = H_i$ for $1 \leq i \leq 4$, and:
%$$\begin{aligned}
%    &\pi^{-1}(C_1) = \mathcal{H}(x_1-x_2), \quad
%     \pi^{-1}(C_2) = \mathcal{H}(x_1+x_2), \\
%    &\pi^{-1}(C_3) = \mathcal{H}(x_3-x_4), \quad
%    \pi^{-1}(C_4) = \mathcal{H}(x_3+x_4),
%\end{aligned}$$
$$\begin{aligned}
    &\pi^{-1}(C_5) = \mathcal{J}(E_{1,2}, E_{3,4}), ~~
     \pi^{-1}(C_6) = \mathcal{J}(E_{2,3}, E_{1,4}), ~~
     \pi^{-1}(C_7) = \mathcal{J}(E_{1,3}, E_{2,4}). \\
\end{aligned}$$
\end{theorem}

\begin{proof}
Recall that $C_1 = \mathcal{V}(M_{12}, M_{13} - M_{23}, M_{14}-M_{24})$. It is clear that $C_1$ is isomorphic to a copy of $\P^2$ contained in the Pl\"ucker hypersurface, $\mathcal{V}(P)$, in $\P^5$. It is well known, see \cite[Exercise 6.5.11, p. 81]{Sha} for example, that such $\P^2$s correspond under the Pl\"ucker embedding $\pi: \mathbb{G} \to \P^5$ to either a pencil of lines passing through a common point in $\P^3$, or to the collection of all lines in a common plane in $\P^3$. We show that the latter possibility is the case here. As described above, we represent an arbitrary line in the plane $\mathcal{V}(x_1 - x_2) \subset \P^3$ by $\ell = \left[\begin{smallmatrix} a_1 & a_1 & a_2 & a_3 \\ b_1 & b_1 & b_2 & b_3 \end{smallmatrix}\right]$. Then $\pi(\ell) = [0: a_1b_2-a_2b_1:a_1b_3-a_3b_1: a_1b_2-a_2b_1: a_1b_3-a_3b_1: a_2b_3-a_3b_2]$, so that $\pi(\ell) \in C_1$. It follows that $\pi^{-1}(C_1) = \mathcal{H}(x_1-x_2)$. The proofs for components $C_2$, $C_3$ and $C_4$ are analogous.

Recall that $C_6 = \mathcal{V}(M_{14} - M_{23}, M_{13}-M_{24}, P)$, and $E_{2,3} = \mathcal{V}(x_1+x_2, x_3-x_4)$, $E_{1,4} = \mathcal{V}(x_1-x_2, x_3+x_4)$. As noted above, $E_{2,3}$ and $E_{1,4}$ are disjoint. Let $p = [a: -a: b: b] \in E_{2,3}$ and $q = [c: c: d: -d] \in E_{1,4}$ be arbitrary points. Then the line containing $p$ and $q$ is represented by $\ell = \left[\begin{smallmatrix} a & -a & b & b \\ c & c & d & -d \end{smallmatrix}\right]$. We have $\pi(\ell) = [2ac: ad-bc: -ad-bc: -ad-bc: ad-bc: -2bd]$ and so $\pi(\ell) \in C_6$.

Conversely, suppose that $\ell = \left[\begin{smallmatrix} a & b & c & d \\ e & f & g & h \end{smallmatrix}\right] \in \pi^{-1}(C_6)$. Then we have 
%\begin{align*}
%&ah-bg+cf-de = 0, \\
%&ag-bh-ce+df = 0.
%\end{align*}
$$\begin{aligned}
&ah-bg+cf-de = 0, \quad 
 ag-bh-ce+df = 0.
\end{aligned}$$
One checks that the line $\ell$ intersects the line $E_{2,3}$ if and only if the 
linear system in variables $(\l, \mu)$ given by
$$\begin{bmatrix} a+b & e+f \\ c-d & g-h \end{bmatrix}
  \begin{bmatrix} \l \\ \mu \end{bmatrix} =
  \begin{bmatrix} 0 \\ 0 \end{bmatrix}$$
has a nontrivial solution. One checks that the determinant of the coefficient 
matrix is zero, so indeed $\ell$ intersects the line $E_{2,3}$. Similarly, one 
also checks that $\ell$ intersects $E_{1,4}$. We conclude that
$\ell \in \mathcal{J}(E_{2,3}, E_{1,4})$. 
Therefore, $\pi^{-1}(C_6) = \mathcal{J}(E_{2,3}, E_{1,4})$. The proofs for 
components $C_5$ and $C_7$ are analogous. 
\end{proof}

 If $p$ is a closed point in $\mathcal{P}(R)$, then there are either three or six $\P^1$s worth of lines in $\pi^{-1}(\mathcal{L}(R))$ passing through $p$. More precisely, we have the following result. 
 %{\color{red} (PG: Change the notation in next result to the $\Sigma_{p,H}$ notation.)}
%(Put in word descriptions for incidence geometry for $R$, $S$, $T$.)}
{\color{blue}
%\subsubsection{Incidence Relations Between Lines and Points}

}

 \begin{theorem} \label{lines through points for R}
Let $p \in \mathcal{P}(R)$ be a closed point. Recall the notation established in Notation \ref{notation for geometry of R}. Suppose that $\{i,j,k,l\} = \{1, 2, 3, 4\}$.
\begin{itemize} 
\item[(1)] If $p \in H_{i} \cap H_{j} = E_{i,j}$ for $1 \leq i < j \leq 4$ and $p \notin \{p_1, p_2, p_3, p_4\}$, then $\mathcal{L}(p)$ consists of:
(a) $\Sigma_{p,H_i}$, (b) $\Sigma_{p,H_j}$, and
(c) the lines in $\mathcal{J}(E_{i,j}, E_{k,l})$ containing $p$.

\item[(2)] If $p = p_i$ for $1 \leq i \leq 4$, then $\mathcal{L}(p)$ consists of: (a) $\Sigma_{p_i, H_j}$, (b) $\Sigma_{p_i, H_k}$, (c) $\Sigma_{p_i, H_l}$,
(d) the lines in $\mathcal{J}(E_{j,k}, E_{i,l})$ containing $p_i$,
(e) the lines in $\mathcal{J}(E_{j,l}, E_{i,k})$ containing $p_i$,
(f) the lines in $\mathcal{J}(E_{k,l}, E_{i,j})$ containing $p_i$.

\end{itemize}

% The points $p_i$, $1 \leq i \leq 4$, on triple intersections of components have six $P^1$s worth of lines passing through them  (Make precise.)
\end{theorem}

\begin{proof}
This follows from the observations in Notation
\ref{notation for geometry of R} and Theorem \ref{lines in P^3 for R}.
\end{proof}

 Finally, consider any line $\ell \in \pi^{-1}(\mathcal{L}(R))$. To conclude the study of the geometry of $R$, we determine how many points $p \in \mathcal{P}(R)$ lie on $\ell$. Of course, if $\ell = E_{i,j}$ for some $1 \leq i < j \leq 4$, then every point on $\ell$ is in $\mathcal{P}(R)$. 

\begin{theorem}\label{points on lines for R}
Let $\ell \in \pi^{-1}(\mathcal{L}(R))$ and $\ell \ne E_{i, j}$ for all $1 \leq i < j \leq 4$. Suppose that $\{i,j,k,l\} = \{1, 2, 3, 4\}$.
\begin{itemize}
    \item[(1)] If $\ell \subset H_i$, then $\ell \cap \mathcal{P}(R)$ consists of either two or three distinct closed points. In the case of two points, one of the points is $p_j$, $p_k$ or $p_l$ while the other point is none of $p_1, p_2, p_3, p_4$. In the case of three points, none of the points are any of $p_1, p_2, p_3, p_4$.
    \item[(2)] If $\ell \in \mathcal{J}(E_{i,j}, E_{k, l})$ and $\ell \not\subset \bigcup_{i = 1}^4 H_i$, then $\ell \cap \mathcal{P}(R)$ consists of two distinct points $q_1 \in E_{i,j}$ and $q_2 \in E_{k,l}$. Moreover, $q_1 \notin \{p_k, p_l\}$ and $q_2 \notin \{p_i, p_j\}$.
\end{itemize}
    
\end{theorem}

\begin{proof}
    For (1), assume that $\ell \subset H_1$. Then $\ell = \mathcal{V}(x_1-x_2, ax_1+bx_2+cx_3+dx_4)$ for some $a, b, c, d \in \k$. We consider four cases: (i) $p_2 \in \ell$, (ii) $p_3 \in \ell$, (iii) $p_4 \in \ell$, (iv) $p_2, p_3, p_4 \not\in \ell$. Suppose (i). Then it follows that $b = -a$ and so $\ell = \mathcal{V}(x_1-x_2, cx_3+dx_4)$. Since $\ell(p_2, p_3) = E_{1,4}$ and $\ell(p_2, p_4) = E_{1,3}$ our assumption that $\ell \ne E_{i,j}$ ensures that $p_3, p_4 \not\in \ell$. It follows that $d \ne \pm c$. The lines $E_{1,2}$, $E_{1,3}$, $E_{1,4}$ are contained in the plane $H_1$, so, by Bezout's theorem, $\ell$ will meet each of these lines in one point. Note that $p_2 \in E_{1,3} \cap E_{1,4}$, so we only need to consider $\ell \cap E_{1,2}$. Note that $\ell \cap E_{1,2} = \{[0:0:d:-c]\}$ and the condition $d \ne \pm c$ ensures that this point is not $p_3$ or $p_4$. We conclude that $\ell \cap \mathcal{P}(R) = \{p_2, [0:0:d:-c]\}$. The cases (ii), (iii) lead to a similar conclusion. Now suppose (iv). It is easy to check that the lines $E_{2,3}$, $E_{2,4}$, $E_{3,4}$ meet the plane $H_1$ in the points $p_4$, $p_3$, $p_2$, respectively. So $\ell$ does not intersect any of these lines. By Bezout's theorem, $\ell \cap (E_{1,2} \cup E_{1,3} \cup E_{1,4}) = \{q_1, q_2, q_3\}$ for some not necessarily distinct points $q_i$. However, if two of these points coincide, then some $q_i$ lies on two of these lines. The three pairwise intersections of $E_{1,2}, E_{1,3}, E_{1,4}$ are the points $p_2, p_3, p_4$, so we contradict the assumption that $p_2, p_3, p_4 \not\in \ell$. We conclude that $\ell \cap \mathcal{P}(R) = \{q_1, q_2, q_3\}$ and that these points are distinct. The proof of (1) for the cases when $\ell \subset H_i$ for $i = 2, 3, 4$ are similar. 

    Now we prove (2). Suppose that $\ell \in \mathcal{J}(E_{1,3}, E_{2,4})$ and $\ell \not\subset \bigcup_{i = 1}^4 H_i$. Write $\ell = \ell(q_1, q_2)$, where $q_1 = [a:a:b:b] \in E_{1,3}$, $q_2 = [c:-c:d:-d] \in E_{2,4}$ for some $a, b, c, d \in \k$. Then $\ell = \mathcal{V}(d(x_1-x_2)-c(x_3-x_4), b(x_1+x_2)-a(x_3+x_4))$. Note that if $abcd = 0$, then $\ell$ is contained in at least one of the $H_i$, contrary to assumption, hence $abcd \ne 0$. We now consider the intersection of $\ell$ with the lines $E_{1,2}$, $E_{1,4}$, $E_{2,3}$, $E_{3,4}$. If $[0:0:e:f] \in E_{1,2}$ is also in $\ell$, then $c(e-f) = 0$ and $a(e+f) = 0$. Since $ac \ne 0$, we have $\ell \cap E_{1,2} = \emptyset$. Similarly, if $[e:e:f:-f] \in E_{1,4}$ is also in $\ell$, then $2cf = 0$ and $2be = 0$. Since $bc \ne 0$, we have $\ell \cap E_{1,4} = \emptyset$. Similar arguments show that $\ell$ does not intersect $E_{2,3}$ or $E_{3,4}$. We conclude that $\ell \cap \mathcal{P}(R) = \{q_1, q_2\}$ and observe that $q_1 \not\in \{p_2, p_4\}$ and $q_1 \not\in \{p_1, p_3\}$. The proof of (2) for the cases when $\ell \in \mathcal{J}(E_{1,2}, E_{3,4}) \cup \mathcal{J}(E_{1,4}, E_{2,3})$ is similar. 
\end{proof}

\begin{remark}\label{Michaela example}
The algebra $R$ is graded by the Klein 4-group $V$. Recall that $H^2(V, \k) = \la [\mu] \ra \cong \Z/2$, where $\mu: V \times V \to \k$ is a certain 2-cocycle. Let $R' = R^{\mu}$ be the twist of $R$ by $\mu$; see \cite{Davies} for more details about this construction. The algebra $R'$ can be presented as:
$$\frac{\k \la x_1, x_2, x_3, x_4 \ra}{\la x_1x_4-x_2x_3,
x_1x_3+x_3x_1,
x_2x_1-x_3x_4,
x_3x_2+x_4x_1,
x_1x_2+x_4x_3,
x_2^2+x_4^2\ra}.$$ 
Since $R$ is AS regular, \cite{Davies} shows that $R'$ is also AS regular. One can check that the point scheme of $R'$ is the union of the line $\mathcal{V}(x_2, x_4)$ and the two planes $\mathcal{V}(x_1+\i x_3)$ and $\mathcal{V}(x_1-\i x_3)$. So the point scheme of $R'$ is the union of a line and a rank two quadric in $\P^3$. 
\end{remark}
\subsection{The point and line schemes of \texorpdfstring{$S$}{S} and \texorpdfstring{$T$}{T}}
 In this subsection we compute the point and line schemes of the algebras $S$ and $T$.  The results and proofs are very similar for these algebras.  Throughout, define the following points of $\P^3$:
 $$e_1 = [1: 0: 0: 0], e_2 = [0: 1: 0: 0], e_3 = [0: 0: 1: 0], e_4 = [0: 0: 0:1].$$ Further, for $0 \leq j,k \leq 3$, we also define the following points of $\P^3$,
 where $\z$ denotes a primitive eighth root of unity in $\k$ and $\i = \z^2$:
$$p_{j,k} = [1:\i^j:\i^k:\i^{-j-k}], \qquad q_{j,k} = [1:\i^j:\z\i^k:\z^3\i^{-j-k}].$$
Note that for $p_{j,k}$ (respectively $q_{j,k}$), the last coordinate is the
(negative of) the inverse of the product of the second and third coordinates.

\begin{theorem} \label{point schemes, automorphisms of S, T}
The point scheme and automorphism $\tau$ for $S$ and $T$ are as 
follows. 
\begin{itemize}
    \item [(1)] Using the presentation of $S$ with generators 
$x_1,x_2,x_3,x_4$ and relations $x_1x_2-x_3^2,x_2x_1-x_4^2,
x_1x_3-x_2x_4,x_2x_3-x_3x_1,x_1x_4-x_4x_2,x_3x_2-x_4x_1$,
the point scheme of $S$ is reduced and consists of twenty 
distinct closed points: 
the $e_i$, $1 \leq i \leq 4$ and $p_{j,k}$ for
$0 \leq j,k \leq 3$ as defined above.

The automorphism $\t: \mathcal{P}(S) \to \mathcal{P}(S)$ is given by:
$\t$ fixes $e_1, e_2$ and interchanges $e_3$ and $e_4$; and
$\t(p_{j,k}) = p_{2k-j,k-j}$, where the subscripts are read
modulo four.  More explicitly, one has
$$\t[1: b: c: b^{-1}c^{-1}] = [1 : b^3c^2 : b^3c : b^2c].$$ 
Hence $\t$ fixes $[1: 1: 1: 1]$ and $[1: 1: -1: -1]$, $\t$ interchanges
$[1:  -1: 1: -1]$ and $[1: -1: -1: 1]$. The remaining twelve points 
partition into three orbits of size four. 

\item [(2)] Using the presentation of $T$ with generators 
$x_1,x_2,x_3,x_4$ and relations $x_1x_2-x_3^2,x_2x_1+x_4^2,
x_1x_3-x_2x_4,x_2x_3-x_3x_1,x_1x_4+x_4x_2,x_3x_2-x_4x_1$,
the point scheme of $T$ is reduced and consists of twenty 
distinct closed points: 
the $e_i$, $1 \leq i \leq 4$ and $q_{j,k}$ for
$0 \leq j,k \leq 3$ as defined above.

The automorphism $\t: \mathcal{P}(T) \to \mathcal{P}(T)$ is given by:
$\t$ fixes $e_1, e_2$ and interchanges $e_3$ and $e_4$; and
$\t(q_{j,k}) = q_{2k+1-j,k-j}$, where the subscripts are read 
modulo four.  More explicitly, one has
$$\t[1: b: c: -b^{-1}c^{-1}] = [1 : b^3c^2 : b^3c : b^2c].$$ 
The points $q_{j,k}$ partition into four orbits of size four.

\end{itemize}
\end{theorem}

\begin{comment}
\begin{itemize}
    \item [(1)] Let $i$ denote a primitive fourth root of unity in $\k$. 
Using the presentation of $S$ with generators 
$x_1,x_2,x_3,x_4$ and relations $x_1x_2-x_3^2,x_2x_1-x_4^2,
x_1x_3-x_2x_4,x_2x_3-x_3x_1,x_1x_4-x_4x_2,x_3x_2-x_4x_1$,
the point scheme of $S$ is reduced and consists of twenty distinct closed points: the $e_i$, $1 \leq i \leq 4$ and $$[1: b: c: d], \text{ where } b, c \in \{1, -1, i,  -i\}, d = b^3 c^{-1}.$$ The automorphism $\t: \mathcal{P}(S) \to \mathcal{P}(S)$ is given by: $\t$ fixes $e_1, e_2$ and interchanges $e_3$ and $e_4$; $$\t[1:  b: c: d] = [1: b^2c d^{-1}:  b^2 d^{-1}:  b d^{-1}].$$ Hence $\t$ fixes $[1: 1: 1: 1]$ and $[1: 1: -1: -1]$, $\t$ interchanges $[1:  -1: 1: -1]$ and $[1: -1: -1: 1]$. The remaining twelve points partition into three orbits of size 4. 
\item[(2)] Let $\z$ denote a primitive eighth root of unity in $\k$. The point scheme of $T$  is reduced and consists of twenty distinct closed points: the $e_i$, $1 \leq i \leq 4$ and $$[1: b: c: d], \text{ where } b \in \{1, -1, \z^2, \z^6\},  c \in \{\z, \z^3, \z^5, \z^7\}, d = -b^3 c^{-1}.$$ The automorphism $\t$ is given by: $\t$ fixes $e_1, e_2$ and interchanges $e_3$ and $e_4$; $$\t[1: b: c: d] = [1: -b^2c d^{-1}: -b^2 d^{-1}: -b d^{-1}].$$ The points $[1: b: c: d]$ partition into four orbits of size 4.
\end{itemize}
\end{comment}

\begin{proof}
We sketch the proof of (1). The proof of (2) is similar. 
With the stated presentation for $S$, the $6 \times 4$ matrix $M$ determined by factoring the relations is:
$$
M = \begin{bmatrix} 0 & x_1 & -x_3 & 0 \\
                      x_2 & 0 & 0 & -x_4 \\
                      0 & 0 & x_1 & -x_2 \\
                      -x_3 & 0 & x_2 & 0 \\
                      0 & -x_4 & 0 & x_1 \\
                      -x_4 & x_3 & 0 & 0 \\
      \end{bmatrix}.
$$
The {\it Macaulay2} function {\tt pointScheme} outputs the fifteen
$4 \times 4$ minors of $M$ as an ideal $I$ of the polynomial ring
$\Q[x_1, x_2, x_3, x_4]$. Then we use the {\it Macaulay2} function
{\tt minimalPrimes} on the ideal $I$ to compute the irreducible components
of the zero locus of $I$. We find that ${\mathcal{P}}(S)$ is the union of
irreducible components: 
$$\begin{array}{ll}
 C_1 = \mathcal{V}(x_2, x_3, x_4), &
 C_2 = \mathcal{V}(x_1, x_3, x_4), \\
 C_3 = \mathcal{V}(x_1, x_2, x_4), &
 C_4 = \mathcal{V}(x_1, x_2, x_3), \\
 C_5 = \mathcal{V}(x_1-x_4, x_2-x_4, x_3-x_4), &
 C_6 = \mathcal{V}(x_1+x_4, x_2+x_4, x_3-x_4), \\
 C_7 = \mathcal{V}(x_1+x_4, x_2-x_4, x_3+x_4), &
 C_8 = \mathcal{V}(x_1-x_4, x_2+x_4, x_3+x_4), \\
 C_{9} = \mathcal{V}(x_1+x_2, x_3-x_4, x_2^2+x_4^2), &
 C_{10} = \mathcal{V}(x_1-x_2, x_3+x_4, x_2^2+x_4^2), \\
 C_{11} = \mathcal{V}(x_1+x_3, x_2-x_4, x_3^2+x_4^2), &
 C_{12} = \mathcal{V}(x_1-x_3, x_2+x_4, x_3^2+x_4^2), \\
 C_{13} = \mathcal{V}(x_1+x_4, x_2-x_3, x_3^2+x_4^2), &
 C_{14} = \mathcal{V}(x_1-x_4, x_2+x_3, x_3^2+x_4^2). 
\end{array}$$
Working over the field $\Q(\i)$, it is then easy to see that the closed points of ${\mathcal{P}}(S)$ are as stated. Van den Bergh proved in unpublished work that if the point scheme of a quadratic AS regular of global dimension four is a finite subscheme of $\P^3$, then it has degree twenty; see \cite{V} for example. Since ${\mathcal P}(S)$ consists of twenty distinct closed points, we conclude that the point scheme of $S$ is reduced.
The determination of the automorphism $\t: {\mathcal P}(S) \to {\mathcal P}(S)$ on closed points is a straightforward computation and is omitted.
\end{proof}

Next we determine the line schemes of $S$ and $T$. We show these schemes are reduced and we find their decompositions into irreducible components. 

\begin{theorem} \label{line scheme of S, T}
Using the presentation for $S$ (respectively $T$) with generators
$x_1, x_2, x_3, x_4$ and relations:
$x_1x_2-x_3^2$, $x_4^2-x_2x_1$ (respectively $x_4^2+x_2x_1$), $x_1x_3-x_2x_4$, 
$x_4x_1-x_3x_2$, $x_2x_3-x_3x_1$, $x_4x_2-x_1x_4$ (respectively $x_4x_2+x_1x_4$), 
the irreducible components of $\mathcal{L}(S)$ and $\mathcal{L}(T)$ are:
$$\begin{array}{ll}
C_1 =    \V(M_{12}, M_{34}, M_{13}-\alpha M_{24}, P),&
C_2 =    \V(M_{12}, M_{34}, M_{13}+\alpha M_{24}, P),\\
C_3 =    \V(M_{13}, M_{24}, M_{14}-\alpha M_{23}, P),&
C_4 =    \V(M_{13}, M_{24}, M_{14}+\alpha M_{23}, P),\\
C_5 =    \V(M_{14}, M_{23}, M_{12}-M_{34}, P),&
C_6 =    \V(M_{14}, M_{23}, M_{12}+M_{34}, P),
\end{array}$$
\vspace*{-.5em}
$$\begin{array}{l}
C_7 =    \V(M_{12}-M_{34}, M_{13}-\alpha M_{24}, M_{14}-\alpha M_{23}, P),\\
C_8 =    \V(M_{12}-M_{34}, M_{13}+\alpha M_{24}, M_{14}+\alpha M_{23}, P),\\
C_9 =    \V(M_{12}+M_{34}, M_{13}-\alpha M_{24}, M_{14}+\alpha M_{23}, P),\\
C_{10} = \V(M_{12}+M_{34}, M_{13}+\alpha M_{24}, M_{14}-\alpha M_{23}, P),\\
\end{array}$$
where $P$ is the Pl\"ucker polynomial, and where $\alpha = 1$ for $S$ and
$\alpha = \i$ for $T$.   For $1 \leq i \leq 10$, the component $C_i$
is isomorphic to $\P^1$. The schemes $\mathcal{L}(S)$
and $\mathcal{L}(T)$ are equidimensional of
dimension one and are reduced schemes.
\end{theorem}

\begin{proof}
%Follow Chandler-Vancliff for reduced part.
We prove the result in the case of $S$, and leave the similar proof in
the case of $T$ to the reader.

Let $B = \k[M_{i,j}: 1 \leq i < j \leq 4]$ be the polynomial ring in the Pl\"ucker coordinates. A computation using the aforementioned {\it Macaulay2} functions shows that the minimal primes of the ideal $I$ of $B$ determining the line scheme $\mathcal{L}(S)$ give the stated components $C_i$, $1 \leq i \leq 10$. Note that each components is cut out by three independent linear forms, so each component is contained in a $\P^2$.  Modulo these linear forms, the Pl\"ucker polynomial gives a smooth conic in the associated $\P^2$, so that $C_i \cong \P^1$. The statement about the dimension of $\mathcal{L}(S)$ follows.

To show that $\mathcal{L}(S)$ is reduced we follow the argument given in \cite[Lemma 3.2, Theorem 3.3]{ChV}, also see \cite[Corollary 3.7]{CSV}. Let $\mathcal{L}'(S) = \mathcal{L}(S)_{\rm{red}}$ be the reduced scheme associated to $\mathcal{L}(S)$.  We recall that $S$ is a noetherian, Auslander-regular domain that satisfies the Cohen-Macaulay property with respect to GK-dimension and has the Hilbert series $(1-t)^{-4}$. Furthermore, we have established that the irreducible components of $\mathcal{L}'(S)$ have dimension one, so the proof of \cite[Lemma 3.2]{ChV} shows that $\mathcal{L}(S)$ has no embedded components. As subschemes of $\P^5$, it is clear that the components $C_i$ for $1 \leq i \leq 10$ each have $\deg(C_i) = 2$, so $\deg \mathcal{L}'(S) = 20$. By \cite[Corollary 3.7]{CSV}, $\deg \mathcal{L}(S) = 20$. Since $\mathcal{L}(S)$ has no embedded points, it follows that $\mathcal{L}(S) = \mathcal{L}'(S)$, so $\mathcal{L}(S)$ is reduced.
\end{proof}

The previous result shows that each component of the line scheme of $S$ and $T$ is a conic obtained by intersecting the Pl\"ucker hypersurface $\mathcal{V}(P)$ with a projective plane in $\P^5$. Such conics are known to correspond, via $\pi: {\mathbb G} \to \P^5$, with one of the two rulings on a quadric in $\P^3$ ; see \cite[Exercise 6.5.13, p. 81]{Sha} for example. In the next result, we describe these quadrics and rulings for each component of the line scheme of $S$ and $T$. We explain how to compute the quadric and ruling by way of an example. Let $C = \V(M_{12}, M_{34}, M_{13}-M_{24}, P)$. A line $\ell$ in $\P^3$, represented as $\left[\begin{smallmatrix} x_1 & x_2 & x_3 & x_4 \\ y_1 & y_2 & y_3 & y_4 \end{smallmatrix}\right]$, is in $\pi^{-1}(C)$ if and only if 
\begin{eqnarray*}
x_1y_2-x_2y_1 & = & 0, \\
x_3y_4-x_4y_3 & = & 0, \\
x_1y_3-x_3y_1-x_2y_4-x_4y_2 & = & 0.
\end{eqnarray*}
%We write this system as a matrix equation
%$$ \begin{bmatrix} -x_2 & x_1 & 0 & 0 \\
%                    0 & 0 & -x_4 & x_3 \\
%                    -x_3 & x_4 & x_1 & -x_2
%    \end{bmatrix} 
%    \begin{bmatrix} y_1 \\  y_2 \\ y_3 \\ y_4 \end{bmatrix}
%    =
%    \begin{bmatrix} 0 \\ 0 \\ 0 \end{bmatrix}.$$
Let $M$ be the coefficient matrix of the above system, viewing the $y_j$ as the variables.  Observe that
$\rank(M) \geq 2$ for $[x_1:x_2:x_3:x_4] \in \P^3$.
%One checks that the $3 \times 3$ minors of $M$ are
%$$x_4(x_1x_3 - x_2x_4), \quad -x_3(x_1x_3 - x_2x_4), \quad x_2(x_1x_3 - x_2x_4), \quad -x_1(x_1x_3 - x_2x_4).$$
One checks that $\rank(M) = 2$ if and only if $x_1x_3-x_2x_4 = 0$.
Define the quadric $Q = \V(x_1x_3-x_2x_4)$, and call this the
{\emph{quadric associated to $C$}}. There are two rulings on $Q$,
namely:
\begin{eqnarray*}
{\mathcal R}_1 & = & \{\V(sx_1+tx_2, tx_3+sx_4) : [s:t] \in \P^1\},\\
{\mathcal R}_2 & = & \{\V(sx_4-tx_1, sx_3-tx_2) : [s:t] \in \P^1\}.
\end{eqnarray*}
It is now straightforward to check that $\pi^{-1}(C) = {\mathcal R}_1$ and
$\pi^{-1}(C) \cap {\mathcal R}_2 = \emptyset$. 

A similar analysis for each component of the line schemes of $S$ and $T$ produces the following tables. Details are left to the reader. 
We find it remarkable that for each quadric $Q$ associated to a component of the line scheme of $S$ or $T$, of the two rulings on $Q$, only one of the rulings produces line modules. The other ruling seems to produce fat point modules of multiplicity two. We discuss our findings and some questions regarding this phenomenon in Section \ref{sect-questions}.

% Components of Line Scheme for S
\begin{table}[H]
\centering
%\begin{tabular}{| c | c | c | c |}
\scalebox{.85}{
\begin{tabular}{| c | c | c | c |}
\hline
{\bf Comp.} & 
%{\bf Lines} &
{\bf Quadric}  & \multicolumn{2}{|c|}{\bf{Ruling}} \\ \hline

$C_1$ & 
%$\begin{pmatrix}t & -s & 0 & 0 \\ 0 & 0 & s & -t\end{pmatrix}$ &
$\V(x_1x_3 - x_2x_4)$ & $sx_1 + tx_2$ & $tx_3 + sx_4$ \\ \hline

$C_2$ &
%$\begin{pmatrix}t & -s & 0 & 0 \\ 0 & 0 & s & t\end{pmatrix}$ &
$\V(x_1x_3 + x_2x_4)$ & $sx_1 + tx_2$ & $tx_3 - sx_4$ \\ \hline

$C_3$ &
%$\begin{pmatrix}t & 0 & -s & 0 \\ 0 & t & 0 & s\end{pmatrix}$ &
$\V(x_1x_4 + x_2x_3)$ & $sx_1 + tx_3$ & $sx_2 - tx_4$ \\ \hline

$C_4$ &
%$\begin{pmatrix}t & 0 & -s & 0 \\ 0 & t & 0 & -s\end{pmatrix}$ &
$\V(x_1x_4 - x_2x_3)$ & $sx_1 + tx_3$&$sx_2 + tx_4$ \\ \hline

$C_5$ &
%$\begin{pmatrix}t & 0 & 0 & -s \\ 0 & s & t & 0\end{pmatrix}$ &
$\V(x_1x_2 + x_3x_4)$ & $sx_1 + tx_4$&$tx_2 - sx_3$ \\ \hline

$C_6$ &
%$\begin{pmatrix}t & 0 & 0 & -s \\ 0 & s & -t & 0\end{pmatrix}$ &
$\V(x_1x_2 - x_3x_4)$ & $sx_1 + tx_4$&$tx_2 + sx_3$ \\ \hline

$C_7$ &
%$\begin{pmatrix}t & -t & -s & -s \\ s & s & t & -t\end{pmatrix}$ &
$\V((x_1^2-x_2^2)+(x_3^2-x_4^2))$ & $s(x_1-x_2) + t(x_3+x_4)$&$ t(x_1+x_2) - s(x_3-x_4)$ \\ \hline

$C_8$ &
%$\begin{pmatrix}t & -t & -s & s \\ s & s & -t & -t\end{pmatrix}$ &
$\V((x_1^2-x_2^2)-(x_3^2-x_4^2))$ & $s(x_1-x_2) + t(x_3-x_4)$&$ t(x_1+x_2) + s(x_3+x_4)$ \\ \hline

$C_9$ &
%$\begin{pmatrix}t & -it & -s & is \\ s & is & t & it\end{pmatrix}$ &
$\V((x_1^2+x_2^2)+(x_3^2+x_4^2))$ &
 $s(x_1+\i x_2) + t(x_3+\i x_4)$&$
   t(x_1-\i x_2) - s(x_3-\i x_4)$ \\ \hline

$C_{10}$ &
%$\begin{pmatrix}t & -it & -s & -is \\ s & is & -t & it\end{pmatrix}$ &
$\V((x_1^2+x_2^2)-(x_3^2+x_4^2))$ &
$s(x_1+\i x_2) + t(x_3-\i x_4)$&$
  t(x_1-\i x_2) + s(x_3+\i x_4)$
\\ \hline
\end{tabular}}
\caption{Quadrics Associated to the Line Scheme of $S$}
\label{table:preimageLSofS}
\end{table}

% Components of line scheme for T
\begin{table}[H]
%\begin{tabular}{| c | c | c | c |}
\scalebox{.85}{
\begin{tabular}{| c | c | c | c |}
\hline
{\bf Comp.} & 
%{\bf Lines} &
{\bf Quadric}  & \multicolumn{2}{|c|}{\bf Ruling} \\ \hline

$C_1$ & 
%$\begin{pmatrix}t & -s & 0 & 0 \\ 0 & 0 & s & it\end{pmatrix}$ &
$\V(x_1x_3 - \i x_2x_4)$ & $sx_1 + tx_2$&$tx_3 + \i sx_4$ \\ \hline

$C_2$ &
%$\begin{pmatrix}t & -s & 0 & 0 \\ 0 & 0 & s & -it\end{pmatrix}$ &
$\V(x_1x_3 + \i x_2x_4)$ & $sx_1 + tx_2$&$tx_3 - \i sx_4$ \\ \hline

$C_3$ &
%$\begin{pmatrix}t & 0 & -s & 0 \\ 0 & t & 0 & is\end{pmatrix}$ &
$\V(x_1x_4 + \i x_2x_3)$ & $sx_1 + tx_3$&$sx_2 + \i tx_4$ \\ \hline

$C_4$ &
%$\begin{pmatrix}t & 0 & -s & 0 \\ 0 & t & 0 & -is\end{pmatrix}$ &
$\V(x_1x_4 - \i x_2x_3)$ & $sx_1 + tx_3$&$sx_2 - \i tx_4$ \\ \hline

$C_5$ &
%$\begin{pmatrix}t & 0 & 0 & -s \\ 0 & s & t & 0\end{pmatrix}$ &
$\V(x_1x_2 + x_3x_4)$ & $sx_1 + tx_4$&$tx_2 - sx_3$ \\ \hline

$C_6$ &
%$\begin{pmatrix}t & 0 & 0 & -s \\ 0 & s & -t & 0\end{pmatrix}$ &
$\V(x_1x_2 - x_3x_4)$ & $sx_1 + tx_4$&$tx_2 + sx_3$ \\ \hline

$C_7$ &
%$\begin{pmatrix}t & -\i t & -s & s \\ s & \i s & -\i t & -\i t\end{pmatrix}$ &
$\V((x_1^2+x_2^2)+\i(x_3^2-x_4^2))$ & $s(x_1+\i x_2) + t(x_3-x_4)$&$ t(x_1-\i x_2) - \i s(x_3+x_4)$ \\ \hline

$C_8$ &
%$\begin{pmatrix}t & -\i t & -s & -s \\ s & \i s & \i t & -\i t\end{pmatrix}$ &
$\V((x_1^2+x_2^2)-\i(x_3^2-x_4^2))$ & $s(x_1+\i x_2) + t(x_3+x_4)$&$ t(x_1-\i x_2) + \i s(x_3-x_4)$ \\ \hline

$C_9$ &
%$\begin{pmatrix}t & t & -s & -\i s \\ s & -s & -\i t & -t\end{pmatrix}$ &
$\V((x_1^2-x_2^2)+\i(x_3^2+x_4^2))$ & $s(x_1+x_2) + t(x_3-\i x_4)$&$ t(x_1-x_2) - \i s(x_3+\i x_4)$ \\ \hline

$C_{10}$ &
%$\begin{pmatrix}t & t & -s & \i s \\ s & -s & \i t & -t\end{pmatrix}$ &
$\V((x_1^2-x_2^2)+\i(x_3^2+x_4^2))$ & $s(x_1+x_2) + t(x_3+\i x_4)$&$ t(x_1-x_2) + \i s(x_3-\i x_4)$ \\ \hline
\end{tabular}}
\caption{Quadrics Associated to the Line Scheme of $T$}
\label{table:preimageLSofT}
\end{table}

\begin{theorem}\label{lines in P^3 for S, T}
    %(Preimages of the components of the line schemes of $S$ and $T$ under $\pi$ goes here. Also quadrics and rulings.)
    The preimages of the components of the line schemes of $S$ and $T$ under $\pi: {\mathbb G} \to \P^5$, along with their associated quadrics and the rulings, are given
    by \Cref{table:preimageLSofS} and \Cref{table:preimageLSofT}, respectively.
\end{theorem}

In Table \ref{table:IntsOfCompsLSofS}, we collect which components
of the line scheme of $S$ intersect, as well as which lines lie in
the intersection; the claims in the table follow from a direct 
calculation, with the equations of the components, followed
by pulling back under the Pl\"ucker embedding.  Component 
intersections that are not listed are empty.  Recall that
$\ell(p,q)$ denotes the line in $\P^3$ passing through the 
points $p$ and $q$.

\begin{table}[H]
\begin{minipage}[t]{.48\textwidth}
\strut\vspace*{-\baselineskip}\newline
\begin{center}
\scalebox{.85}{
\begin{tabular}{|c | c | c |}
\hline
{\bf Intersections} & \multicolumn{2}{c|}{{\bf Corresponding lines}} \\
\hline
$C_1 \cap C_2$    & $\ell(e_1,e_4)$ & $\ell(e_2,e_3)$ \\ \hline
$C_1 \cap C_7$    & $\ell(p_{0,0},p_{0,2})$ & $\ell(p_{2,0},p_{2,2})$ \\ \hline
$C_1 \cap C_9$    & $\ell(p_{1,0},p_{1,2})$ & $\ell(p_{3,0},p_{3,2})$ \\ \hline
$C_2 \cap C_8$    & $\ell(p_{0,1},p_{0,3})$ & $ \ell(p_{2,1},p_{2,3})$ \\ \hline
$C_2 \cap C_{10}$ & $\ell(p_{1,1},p_{1,3})$ & $ \ell(p_{3,1},p_{3,3})$ \\ \hline
$C_3 \cap C_4$    & $\ell(e_1, e_2)$ & $ \ell(e_3,e_4)$ \\ \hline
$C_3 \cap C_7$    & $\ell(p_{0,1},p_{2,1})$ & $ \ell(p_{0,3},p_{2,3})$ \\ \hline
$C_3 \cap C_{10}$ & $\ell(p_{1,0},p_{3,0})$ & $ \ell(p_{1,2},p_{3,2})$ \\ \hline
\end{tabular}}
\end{center}
\end{minipage}
\begin{minipage}[t]{.48\textwidth}
\strut\vspace*{-\baselineskip}\newline
\begin{center}
\scalebox{.85}{
\begin{tabular}{|c | c | c |}
\hline
{\bf Intersections} & \multicolumn{2}{c|}{{\bf Corresponding lines}} \\
\hline
$C_4 \cap C_8$    & $\ell(p_{0,2},p_{2,2})$ & $ \ell(p_{0,0},p_{2,0})$ \\ \hline
$C_4 \cap C_9$    & $\ell(p_{1,1},p_{3,1})$ & $ \ell(p_{1,3},p_{3,3})$ \\ \hline
$C_5 \cap C_6$    & $\ell(e_1, e_3)$ & $ \ell(e_2,e_4)$ \\ \hline
$C_5 \cap C_7$    & $\ell(p_{1,3},p_{3,1})$ & $ \ell(p_{1,1},p_{3,3})$ \\ \hline
$C_5 \cap C_8$    & $\ell(p_{1,2},p_{3,0})$ & $ \ell(p_{1,0},p_{3,2})$ \\ \hline
$C_6 \cap C_9$    & $\ell(p_{0,1},p_{2,3})$ & $ \ell(p_{0,3},p_{2,1})$ \\ \hline
$C_6 \cap C_{10}$ & $\ell(p_{0,0},p_{2,2})$ & $ \ell(p_{0,2},p_{2,0})$ \\ \hline
\end{tabular}}
\end{center}
\end{minipage}
\caption{Intersections of Components of the Line Scheme of $S$}
\label{table:IntsOfCompsLSofS}
\end{table}
It is straightforward (but cumbersome) to check that the
lines appearing in the above table are the only lines 
of $S$ which contain points of $S$, and each such 
line contains precisely two points of $S$. Indeed, 
each pair of points in $\{e_1,e_2,e_3,e_4\}$
lies on a unique line of $S$, and the points 
$p_{j,k}$ and $p_{j',k'}$ lie on a line of $S$ if
and only if $j - j'$ and $k - k'$ are even, and this line
is unique.

% in "standard form" but points used to describe
% the elements are not in point scheme.
\begin{comment}
\begin{center}
\begin{tabular}{|c | c | c |}
\hline
{\bf Pairwise intersections} & \multicolumn{2}{c|}{{\bf Corresponding lines}} \\
\hline
$C_1 \cap C_2$    & $\ell(e_1,e_4)$ & $\ell(e_2,e_3)$ \\ \hline
$C_1 \cap C_7$    & $\ell(p_{12}(1), p_{34}(1))$ & $\ell(p_{12}(-1),p_{34}(-1))$ \\ \hline
$C_1 \cap C_9$    & $\ell(p_{12}(i), p_{34}(-i))$ & $ \ell(p_{12}(-i),p_{34}(i))$ \\ \hline
$C_2 \cap C_8$    & $\ell(p_{12}(-1), p_{34}(1))$ & $ \ell(p_{12}(1),p_{34}(-1))$ \\ \hline
$C_2 \cap C_{10}$ & $\ell(p_{12}(i), p_{34}(i))$ & $ \ell(p_{12}(-i),p_{34}(-i))$ \\ \hline
$C_3 \cap C_4$    & $\ell(e_1, e_2)$ & $ \ell(e_3,e_4)$ \\ \hline
$C_3 \cap C_7$    & $\ell(p_{13}(i), p_{24}(-i))$ & $ \ell(p_{13}(-i),p_{24}(i))$ \\ \hline
$C_3 \cap C_{10}$ & $\ell(p_{13}(1), p_{24}(-1))$ & $ \ell(p_{13}(-1),p_{24}(1))$ \\ \hline
$C_4 \cap C_8$    & $\ell(p_{13}(1), p_{24}(1))$ & $ \ell(p_{13}(-1),p_{24}(-1))$ \\ \hline
$C_4 \cap C_9$    & $\ell(p_{13}(i), p_{24}(i))$ & $ \ell(p_{13}(-i),p_{24}(-i))$ \\ \hline
$C_5 \cap C_6$    & $\ell(e_1, e_3)$ & $ \ell(e_2,e_4)$ \\ \hline
$C_5 \cap C_7$    & $\ell(p_{14}(1), p_{23}(-1))$ & $ \ell(p_{14}(-1),p_{23}(1))$ \\ \hline
$C_5 \cap C_8$    & $\ell(p_{14}(i), p_{23}(i))$ & $ \ell(p_{14}(-i),p_{23}(-i))$ \\ \hline
$C_6 \cap C_9$    & $\ell(p_{14}(i), p_{23}(-i))$ & $ \ell(p_{14}(-i),p_{23}(i))$ \\ \hline
$C_6 \cap C_{10}$ & $\ell(p_{14}(1), p_{23}(1))$ & $ \ell(p_{14}(-1),p_{23}(-1))$ \\ \hline
\end{tabular}
\end{center}
\end{comment}

In Table \ref{table:IntsOfCompsLSofT}, we collect the corresponding data
for the components of the line scheme of $T$ as well.  Note that
in the order we have chosen, the same components intersect,
and each intersection contains precisely two points as before.
Also, one may check that the points and lines of $T$ satisfy similar
relations to those laid out in the previous paragraph for $S$.

\begin{table}[H]
\begin{minipage}[t]{.48\textwidth}
\strut\vspace*{-\baselineskip}\newline
\begin{center}
\scalebox{.85}{
\begin{tabular}{|c | c | c |}
\hline
{\bf Intersections} & \multicolumn{2}{c|}{{\bf Corresponding lines}} \\
\hline
$C_1 \cap C_2$    & $\ell(e_1,e_4)$ & $\ell(e_2,e_3)$ \\ \hline
$C_1 \cap C_7$    & $\ell(q_{3,1},q_{3,3})$ & $\ell(q_{1,1},q_{1,3})$ \\ \hline
$C_1 \cap C_9$    & $\ell(q_{0,1},q_{0,3})$ & $\ell(q_{2,1},q_{2,3})$ \\ \hline
$C_2 \cap C_8$    & $\ell(q_{1,0},q_{1,2})$ & $\ell(q_{3,0},q_{3,2})$ \\ \hline
$C_2 \cap C_{10}$ & $\ell(q_{2,0},q_{2,2})$ & $\ell(q_{0,0},q_{0,2})$ \\ \hline
$C_3 \cap C_4$    & $\ell(e_1, e_2)$ & $ \ell(e_3,e_4)$ \\ \hline
$C_3 \cap C_7$    & $\ell(q_{1,2},q_{3,2})$ & $\ell(q_{1,0},q_{3,0})$ \\ \hline
$C_3 \cap C_{10}$ & $\ell(q_{0,1},q_{2,1})$ & $\ell(q_{0,3},q_{2,3})$ \\ \hline
\end{tabular}}
\end{center}
\end{minipage}
\begin{minipage}[t]{.48\textwidth}
\strut\vspace*{-\baselineskip}\newline
\begin{center}
\scalebox{.85}{
\begin{tabular}{|c | c | c |}
\hline
{\bf Intersections} & \multicolumn{2}{c|}{{\bf Corresponding lines}} \\
\hline
$C_4 \cap C_8$    & $\ell(q_{1,3},q_{3,3})$ & $\ell(q_{1,1},q_{3,1})$ \\ \hline
$C_4 \cap C_9$    & $\ell(q_{0,0},q_{2,0})$ & $\ell(q_{0,2},q_{2,2})$ \\ \hline
$C_5 \cap C_6$    & $\ell(e_1, e_3)$ & $ \ell(e_2,e_4)$ \\ \hline
$C_5 \cap C_7$    & $\ell(q_{0,2},q_{2,0})$ & $\ell(q_{0,0},q_{2,2})$ \\ \hline
$C_5 \cap C_8$    & $\ell(q_{0,3},q_{2,1})$ & $\ell(q_{0,1},q_{1,3})$ \\ \hline
$C_6 \cap C_9$    & $\ell(q_{1,0},q_{3,2})$ & $\ell(q_{1,2},q_{3,0})$ \\ \hline
$C_6 \cap C_{10}$ & $\ell(q_{1,3},q_{3,1})$ & $\ell(q_{1,1},q_{3,3})$ \\ \hline
\end{tabular}}
\end{center}
\end{minipage}
\caption{Intersections of Components of the Line Scheme of $T$}
\label{table:IntsOfCompsLSofT}
\end{table}

The following theorems summarize the incidence relations
for $S$ and $T$:
\begin{theorem} \label{thm:incidenceS}
The point scheme and line scheme of $S$ satisfy:
\begin{itemize}
    \item[(1)] For each point $p$ of $S$, there are precisely
    three lines of $S$ which contain $p$.
    \item[(2)] When $p = e_j$, these
    are the lines joining $e_j$ and
    $\{e_1,e_2,e_3,e_4\} \setminus \{e_j\}$;
    when $p = p_{j,k}$, these are the lines joining $p$ 
    with the points $p_{j+2,k},p_{j,k+2},p_{j+2,k+2}$,
    where subscripts are read modulo four.
    \item[(3)] The lines described above are exactly
    the lines that lie in precisely two components
    of the line scheme of $S$.  Lines that are contained
    in only one component of the line scheme do not contain
    any points of $S$.  There are precisely thirty lines of
    $S$ which contain points of $S$.
    \item[(4)] Each line of $S$ containing a point of $S$
    contains precisely two points of $S$.
\end{itemize}
\end{theorem}

\begin{theorem} \label{thm:incidenceT}
The point scheme and line scheme of $T$ satisfy the same 
conclusions as Theorem \ref{thm:incidenceS} after changing $S$ 
to $T$, and $p_{j,k}$ to $q_{j,k}$.
\end{theorem}

% https://q.uiver.app/#q=WzAsMTAsWzQsMCwiQ18xIl0sWzgsMiwiQ185Il0sWzAsMiwiQ18yIl0sWzQsMiwiQ183Il0sWzYsMywiQ182Il0sWzUsNSwiQ18zIl0sWzMsNSwiQ181Il0sWzIsMywiQ197MTB9Il0sWzIsNywiQ184Il0sWzYsNywiQ180Il0sWzIsMCwiUF80IiwxLHsic3R5bGUiOnsiaGVhZCI6eyJuYW1lIjoibm9uZSJ9fX1dLFsyLDgsIlBfNSIsMSx7InN0eWxlIjp7ImhlYWQiOnsibmFtZSI6Im5vbmUifX19XSxbOCw5LCJQXzEiLDEseyJzdHlsZSI6eyJoZWFkIjp7Im5hbWUiOiJub25lIn19fV0sWzksMSwiUF8yIiwxLHsic3R5bGUiOnsiaGVhZCI6eyJuYW1lIjoibm9uZSJ9fX1dLFsxLDAsIlBfMyIsMSx7InN0eWxlIjp7ImhlYWQiOnsibmFtZSI6Im5vbmUifX19XSxbMCwzLCJQXzEiLDEseyJzdHlsZSI6eyJoZWFkIjp7Im5hbWUiOiJub25lIn19fV0sWzIsNywiUF8yIiwxLHsic3R5bGUiOnsiaGVhZCI6eyJuYW1lIjoibm9uZSJ9fX1dLFs4LDYsIlBfMyIsMSx7InN0eWxlIjp7ImhlYWQiOnsibmFtZSI6Im5vbmUifX19XSxbNSw5LCJQXzQiLDEseyJzdHlsZSI6eyJoZWFkIjp7Im5hbWUiOiJub25lIn19fV0sWzQsMSwiUF81IiwxLHsic3R5bGUiOnsiaGVhZCI6eyJuYW1lIjoibm9uZSJ9fX1dLFs3LDQsIlBfMSIsMSx7InN0eWxlIjp7ImhlYWQiOnsibmFtZSI6Im5vbmUifX19XSxbNyw1LCJQXzMiLDEseyJsYWJlbF9wb3NpdGlvbiI6NjAsInN0eWxlIjp7ImhlYWQiOnsibmFtZSI6Im5vbmUifX19XSxbMyw2LCJQXzIiLDEseyJzdHlsZSI6eyJoZWFkIjp7Im5hbWUiOiJub25lIn19fV0sWzMsNSwiUF81IiwxLHsic3R5bGUiOnsiaGVhZCI6eyJuYW1lIjoibm9uZSJ9fX1dLFs0LDYsIlBfNCIsMSx7ImxhYmVsX3Bvc2l0aW9uIjo2MCwic3R5bGUiOnsiaGVhZCI6eyJuYW1lIjoibm9uZSJ9fX1dXQ==

\section{Remarks and Questions} 
\label{sect-questions}

In this final section we discuss some questions and problems that we intend to study in forthcoming work. 
\subsection{Dual reflection groups and their associated AS regular algebras}

This work began with the search for groups $G$ of order 16 for which there is an associated $G$-graded AS regular algebra $A$ where the identity component $A_e$ of $A$ is also AS regular.  The problem of classifying all finite dual reflection groups, groups $G$ for which such a $G$-graded AS regular algebra exists, remains.  Are there infinite families of dual reflection groups, beyond those in Examples \ref{Craw} and \ref{wreath}? In this paper the relations for the algebras $A$ associated to the modular and semidihedral groups of order 16  were obtained from the multiplication table of the group grades of the degree one elements of $A$. Do the algebras associated to dual reflection groups always have binomial relations?  Are there properties particular to AS regular algebras associated to dual reflection groups? More generally, do AS regular algebras graded by a finite group have any special algebraic or geometric properties?

\subsection{Rulings on quadrics and fat points}

Let $A$ denote the algebra $S$ or $T$. Recall that the irreducible components of the line scheme of $A$, the $C_i$, $1 \leq i \leq 10$, are ten smooth conics in $\P^5$. Moreover, to each $C_i$ we have associated a unique smooth quadric $Q_i$ in $\P^3$. Of the two rulings on $Q_i$, we have proved that only one of the rulings produces line modules. We call this {\it{ruling one}}, and denote it by ${\mathcal{R}}_{i, 1}$. The other ruling will be denoted by ${\mathcal{R}}_{i, 2}$ and be referred to as {\it{ruling two}} (on $Q_i)$. 

Let $\ell \in {\mathcal{R}}_{i,2}$ and assume that $\ell \notin {\mathcal{R}}_{j, 1}$ for any $1 \leq j \leq 10$. Then computer evidence, using \emph{Macaulay2}, indicates that the module $M(\ell)$ has Hilbert series $1 + \frac{2t}{1-t}$. So the module $F(\ell) = [M(\ell)(1)]_{\geq 1}$ would have Hilbert series $\frac{2}{1-t}$. Recall that a \emph{fat point module of multiplicity $e$} is a right $A$-module that is generated in degree $0$, is critical with respect to GK-dimension, and has Hilbert series $\frac{e}{1-t}$. 

\begin{question}
Let $\ell \in {\mathcal{R}}_{i,2}$ and assume that $\ell \notin {\mathcal{R}}_{j, 1}$ for any $1 \leq j \leq 10$. Is $F(\ell)$ always a fat point module of multiplicity two?
\end{question}

We have proved that the answer is affirmative for some cases. 

\subsection{Graded isolated singularities and rulings}

Let $A$ be a connected graded noetherian $\k$-algebra. Let $\gr A$ denote the category of finitely generated graded right $A$-modules, and $\tors A$ be the full subcategory of $\gr A$ consisting of finite-dimensional $\k$-modules. Then the noncommutative projective scheme associated to $A$ is the Serre quotient $\Proj A := \dfrac{\gr A}{\tors A}.$  The \emph{global dimension} of $\Proj A$ is:
$$\gldim(\Proj A) := \sup\{i : \Ext^i_{\Proj A}(\mathcal{M},\mathcal{N}) \ne 0 \text{ for some } \mathcal{M}, \mathcal{N} \in \Proj A\}.$$
Ueyama, \cite[Definition 2.2]{U}, defines $A$ to be a \emph{graded isolated singularity} if $\Proj A$ has finite global dimension. When $S$ is an AS regular algebra and $z \in S_2$ is a central (or normal) regular element, it is interesting to study the \emph{noncommutative quadric hypersurface} $\Proj A$, where $A = S/\la z \ra$, and to determine if $A$ is a graded isolated singularity. 

Smith and Van den Bergh \cite{Smith-vdB} study a notion of \emph{rulings} associated to noncommutative quadric hypersurfaces $\Proj A$, where $A = S/\la z \ra$. Let $\MCM(A)$ be the category of maximal Cohen-Macaulay modules over $A$. They define $$\mathbb{M} = \{M \in \MCM(A) : M \text{ is indecomposable}, M_0 \in \k^2, M = M_0A\}.$$ For $M \in \mathbb{M}$, the \emph{ruling corresponding to $M$} is: $$\{\coker \phi : \phi \in \P(\Hom_{\gr A}(M(-1), A))\}.$$ Under the hypothesis that $A$ is a domain, \cite[Section 7.2]{Smith-vdB} proves that each ruling consists of a $\P^1$s worth of line modules of $A$, and $\Proj A$ has two rulings when $A$ is a graded isolated singularity, otherwise it has one ruling. 

We now discuss the above general results in the context of the algebra $S$ (see \ref{definition S}) of this paper. Recall that the algebra $S$ has a single quadratic central element (up to scale), $z = x_1^2+x_2^2+x_3x_4+x_4x_3$. 

\begin{proposition}\label{central hypersurface}
Let $S$ be the algebra defined in \ref{definition S}. Let $z = x_1^2+x_2^2+x_3x_4+x_4x_3 \in S_2$ and let $A = S/\la z \ra$. Then: 
\begin{itemize}
\item[(1)] $A$ is not a domain,
\item[(2)] $A$ is a graded isolated singularity.
\end{itemize}
\end{proposition}

\begin{proof}
    For (1), one checks that in $S$, $z = (x_1-x_2-x_3-x_4)^2.$
    For (2), let $w = x_1^2 \in A^{!}_2$. Then $w$ is a central regular element of $A^{!}$ and $A^{!}/\la w \ra \cong S^{!}$. The algebra $C(A) = A^{!}[w^{-1}]_0$ is isomorphic to the algebra $\k^4 \oplus M_2(\k)$; in particular, $C(A)$ is semisimple. By \cite[Proposition 5.2(2)]{Smith-vdB} $\Proj A$ has finite global dimension. 
\end{proof}

Based on the fact that the algebra $C(A)$ in the previous proof is isomorphic to $\k^4 \oplus M_2(\k)$, the set $\mathbb{M}$ should be a singleton. 

\begin{question}
Let $S$ be the algebra defined in \ref{definition S}. Let $z = x_1^2+x_2^2+x_3x_4+x_4x_3 \in S_2$ and let $A = S/\la z \ra$. What is the ruling corresponding to each $M \in \mathbb{M}$?
\end{question}

\begin{question} Referring to Table \ref{table:preimageLSofS}, the line modules corresponding to component $C_9$ are all annihilated by the central element $z = x_1^2+x_2^2+x_3x_4+x_4x_3 \in S$. Hence they are line modules for the quotient $A = S/\la z \ra$.  What is the relationship between the notion of ruling defined by \cite{Smith-vdB} and the ruling given in Table 1?
\end{question}

The paper \cite{Smith-vdB} studies the case of a central element $z \in S_2$, and many of the results on rulings are proved under the hypothesis that the algebra $A = S/\la z \ra$ is a domain. So we have the following general question.

\begin{question}
Suppose that $S$ is a Koszul AS regular algebra of global dimension four, and $w \in S_2$ is a normal regular element. Let $A = S/\la w \ra$. Is there a good notion of rulings on the noncommutative projective quadric $\Proj A$? If so, what is the relationship between such rulings and line modules over $A$, or $S$?
\end{question}

Assuming that the last questions have reasonable answers, in the case of the algebra $S$ (see \ref{definition S}), we have the following.

\begin{question} The algebra $S$ has three non-central linearly independent quadratic normal elements: $\omega_i$, $1 \leq i \leq 3$. Let $A_i = S/\la w_i \ra$. Is $A_i$ a graded isolated singularity? Referring to Table \ref{table:preimageLSofS}, the line modules corresponding to components, $C_5$, $C_6$ and $C_{10}$ are annihilated by $w_i$, $1 \leq i \leq 3$, respectively.  What is the relationship between rulings defined via MCMs and the rulings given in Table \ref{table:preimageLSofS}?
\end{question}

% We will call $A_i$, or more precisely the noncommutative projective scheme $\Proj A$, a \emph{noncommutative quadric hypersurface in $\Proj S$.} Such objects have been studied in \cite{Smith-vdB}. In that work, a homological notion of ruling is defined on noncommutative quadrics and one of the main results is that a noncommutative quadric is smooth if and only if each of the two rulings produces \emph{noncommutative lines}, that is, line modules. In our case, we have checked that line modules corresponding to precisely four components of the line scheme are annihilated by each $w_i$, respectively. 

% \begin{question}
% Is there a precise relationship between the results of \cite{Smith-vdB} and our results?
% \end{question}

\bibliographystyle{plain}
\bibliography{bibliog2}

\end{document}